\documentclass[11pt]{amsart}

\usepackage[utf8]{inputenc}
\usepackage[T1]{fontenc}

\usepackage{amsmath,amssymb,amsfonts,textcomp,amsthm,xifthen,psfrag,graphicx,color,pgfplots}
\usepackage{ifthen}

\usepackage{enumerate}

\usepackage{fullpage}

\usepackage[utf8]{inputenc}

\usepackage[pdftex,
            pdfauthor={Thomas F\"uhrer, Norbert Heuer, Ernst P. Stephan},
            pdftitle={DPG for Signorini problem},
            ]{hyperref}

\newcommand{\eps}{\varepsilon}

\newtheorem{theorem}{Theorem}
\newtheorem{lemma}[theorem]{Lemma}

\newtheorem{remark}[theorem]{Remark}
\theoremstyle{definition}

\DeclareMathOperator{\ext}{\mathcal{E}}

\newcommand{\dual}[2]{\langle#1\hspace*{.5mm},#2\rangle}
\newcommand{\ip}[2]{(#1\hspace*{.5mm},#2)}

\newcommand{\norm}[3][]{#1\|#2#1\|_{#3}}

\def\enorm#1{|\hspace*{-.5mm}|\hspace*{-.5mm}|#1|\hspace*{-.5mm}|\hspace*{-.5mm}|}

\def\pwnabla{\nabla_\TT}

\def\div{{\rm div\,}}
\def\pwdiv{ {\rm div}_{\TT}\,}
\def\pwlap{ {\Delta}_{\TT}\,}

\newcommand{\trace}{\gamma}

\newcommand{\ran}{\operatorname{ran}}

\newcommand{\set}[2]{\left\lbrace{#1} \,:\, {#2} \right\rbrace}

\newcommand{\R}{\ensuremath{\mathbb{R}}}
\newcommand{\N}{\ensuremath{\mathbb{N}}}
\newcommand{\HH}{\ensuremath{\boldsymbol{H}}}

\newcommand{\nn}{{\boldsymbol{n}}}
\newcommand{\vv}{\ensuremath{\boldsymbol{v}}}
\newcommand{\ww}{\ensuremath{\boldsymbol{w}}}
\newcommand{\TT}{\ensuremath{\mathcal{T}}}

\newcommand{\cS}{\ensuremath{\mathcal{S}}}

\newcommand{\el}{\ensuremath{T}}

\newcommand{\OO}{\ensuremath{\mathcal{O}}}

\newcommand{\ssigma}{{\boldsymbol\sigma}}
\newcommand{\ttau}{{\boldsymbol\tau}}
\newcommand{\llambda}{{\boldsymbol\lambda}}
\newcommand{\qq}{{\boldsymbol{q}}}
\newcommand{\uu}{\boldsymbol{u}}

\newcommand{\err}{\operatorname{err}}

\begin{document}

\title[DPG for Signorini problem]{On the DPG method for Signorini problems}
\date{\today}

\author{Thomas F\"uhrer}
\author{Norbert Heuer}
\address{Facultad de Matem\'{a}ticas,
Pontificia Universidad Cat\'{o}lica de Chile,
Vicku\~{n}a Mackenna 4860, Santiago, Chile}
\email{\{tofuhrer,nheuer\}@mat.uc.cl}

\author{Ernst P. Stephan}
\address{Institut f\"ur Angewandte Mathematik,
Leibniz Universit\"at Hannover,
Welfengarten 1, 30167 Hannover, Germany}
\email{stephan@ifam.uni-hannover.de}

\thanks{{\bf Acknowledgment.} 
Supported by CONICYT through FONDECYT projects 1150056 and 3150012}

\begin{abstract}
We derive and analyze discontinuous Petrov-Galerkin methods with optimal test functions for Signorini-type
problems as a prototype of a variational inequality of the first kind. We present different symmetric and
non-symmetric formulations where optimal test functions are only used for the PDE part of the problem,
not the boundary conditions. For the symmetric case and lowest order approximations,
we provide a simple a posteriori error estimate. In a second part, we apply our technique to the
singularly perturbed case of reaction dominated diffusion.
Numerical results show the performance of our method and, in particular,
its robustness in the singularly perturbed case.
\end{abstract}
\maketitle


\section{Introduction}

In this paper we develop a framework to solve contact problems by the discontinuous Petrov-Galerkin method with optimal
test functions (DPG method). We consider a simplified model problem ($-\Delta u+u=f$ in a bounded domain)
with unilateral boundary conditions that resembles a scalar version of the Signorini problem in linear elasticity
\cite{Signorini_33_SAQ,Fichera_71_UCE}. We prove well-posedness of our formulation and quasi-optimal convergence.
We then illustrate how the scheme can be adapted to the singularly perturbed case of reaction-dominated diffusion
($-\eps\Delta u + u=f$). Specifically, we use the DPG setting from \cite{HeuerK_RDM} to design a method that controls
the field variables ($u$, $\nabla u$, and $\Delta u$) in the so-called \emph{balanced norm}
(which corresponds to $\|u\|+\eps^{1/4}\|\nabla u\|+\eps^{3/4}\|\Delta u\|$ with $L^2$-norm $\|\cdot\|$),
cf.~\cite{LinS_12_BFE}.
The balanced norm is stronger than the energy norm that stems from the Dirichlet bilinear form of the problem.

There is a long history of numerical analysis for contact problems, and more generally for variational
inequalities, cf. the books by Glowinski, Lions, Tr{\'e}moli{\`e}res~\cite{GlowinskiLT_81_NAV} and
Kikuchi, Oden~\cite{KikuchiO_88_CPE}. Some early papers on the (standard and mixed) finite element method
for Signorini-type problems are \cite{HlavacekL_77_FEM,BrezziHR_77_EEF,ScarpiniV_77_EEA,BrezziHR_78_EEF,Loppin_78_PCS}.
Unilateral boundary conditions usually give rise to limited regularity of the solution and authors
have made an effort to establish optimal convergence rates of the finite element method, see, e.g.,
\cite{ChoulyH_13_NMU,DrouetH_15_OCD}. 
Recently there has been some interest to extend other numerical schemes to unilateral contact problems,
e.g., least squares \cite{AttiaCS_09_FOS}, the local discontinuous Galerkin method \cite{BustinzaS_12_EEL}, and
a Nitsche-based finite element method \cite{ChoulyH_13_NMU}.

Our objective is to set up an appropriate framework to apply DPG technology to variational inequalities, in this
case Signorini-type problems. At this point we do not intend to be more competitive than previous schemes.
In particular, we are not concerned with the reduced regularity of solutions that limit convergence orders.
In any case, we show well-posedness of our scheme for the minimum regularity of $u$ in $H^1$ and right-hand side function
$f$ in $L^2$. Our primary focus is to use DPG schemes in such a way that their performance is not reduced
by the presence of Signorini boundary conditions. The DPG method aims at ensuring discrete inf-sup stability
by the choice of norms and test functions, cf.~\cite{DemkowiczG_11_ADM,DemkowiczG_11_CDP}.
This is particularly important for singularly perturbed problems where one can achieve robustness of
error control (the discrete inf-sup constant does not depend on perturbation parameters),
see~\cite{DemkowiczH_13_RDM,NiemiCC_13_ASD,ChanHBTD_14_RDM,BroersenS_14_RPG,HeuerK_RDM}.

Recently we have found a setting for the coupling of DPG and Galerkin boundary elements (BEM)
to solve transmission problems, see~\cite{FuehrerHK_CDB}. The principal idea is to take a variational
formulation of the interior problem (suitable for the DPG method) and to add transmission conditions
as a constraint. This constraint can be given as boundary integral equations in a least-squares or Galerkin form.
In this paper, we follow this very strategy. The partial differential equation is put in variational form
(without considering any boundary condition) and the Signorini conditions are added as constraint.
It turns out that the whole scheme can be written as a variational inequality of the first kind
where only the PDE part is tested with \emph{optimal test functions}, as is DPG strategy.
Then, proving coercivity and boundedness
of the bilinear form, the Lions-Stampacchia theorem proves well-posedness.

Let us note that there are finite element/boundary element coupling schemes for contact problems,
see~\cite{CarstensenG_97_FBC,MaischakS_05_FBC,GaticaMS_11_NAT}.
However, their coupling variants generalize the setting of contact problems on bounded domains.
In contrast, our coupling scheme \cite{FuehrerHK_CDB} (for a standard transmission problem)
separates the PDE on the bounded domain from the transmission conditions (and exterior problem)
in such a way that they can be formulated as a variational PDE plus constraint. As we have mentioned,
this is critical for combining a DPG method with transmission conditions as in \cite{FuehrerHK_CDB},
or with contact conditions as we show here.

The singularly perturbed case is more technical for two reasons. First, the DPG setting itself for the PDE
is more complicated (we use a robust formulation with three field variables)
and, second, the combination of PDE and Signorini conditions has to take into account the diffusion coefficient
as scaling parameter. In that way we apply what we have learnt from the DPG scheme \cite{HeuerK_RDM}
for reaction dominated diffusion, and from the DPG-BEM coupling \cite{FuehrerH_RCD}
for this case. To the best of our knowledge, this is the first mathematical analysis of a numerical
scheme for a singularly perturbed contact problem.

We will see that there are symmetric and non-symmetric forms to include the Signorini boundary conditions.
In the symmetric case we are able to provide a simple a posteriori error estimate that is based on the
DPG energy error. Let us also remark that we focus on ultra-weak variational formulations. This has the advantage
that both the trace and the flux appear as independent unknowns. This allows for a symmetric formulation of
the Signorini conditions, needed for our a posteriori error bound. Other variational formulations
can be considered analogously and they give rise to non-symmetric well-posed formulations. In those cases,
however, we have no a posteriori error analysis. Let us also mention that our scheme and
analysis of the scalar Signorini problem
can be extended to variational inequalities of the second kind, e.g., including Coulomb friction
and to linear elasticity so that, indeed, the Signorini problem can be solved by our DPG scheme. 

In the following we continue this introduction by presenting our model problem, by introducing an abstract
formulation as variational inequality of the first kind that is suitable for the DPG scheme,
and by giving a final overview of the remainder of this paper.

\subsection*{Model problem}
Let $\Omega\subset\R^d$ ($d\in\{2,3\}$) be a simply connected Lipschitz domain with
boundary $\Gamma=\partial\Omega$ and unit normal vector $\nn$ on $\Gamma$ pointing towards $\R^d\setminus\overline\Omega$.
For $f\in L^2(\Omega)$ we consider the model problem
\begin{subequations}\label{eq:intro:sigpde}
\begin{align}\label{eq:intro:signorini}
  -c\Delta u + u &= f \quad\text{in }\Omega
\end{align}
subject to the Signorini boundary conditions
\begin{align}\label{eq:intro:bc}
  u \geq 0, \quad \frac{\partial u}{\partial \nn}\geq 0, \quad u \frac{\partial u}{\partial \nn} = 0
  \quad\text{on}\quad\Gamma.
\end{align}
\end{subequations}
Initially we will study the case of constant $c=1$. Later we will illustrate the applicability of our
DPG scheme to the singularly perturbed problem with constant $c=\eps$, assuming $\eps$ to be a small positive
number.
Of course, the restriction of $u$ to $\Gamma$ is understood in the sense of the trace, and its normal derivative
on $\Gamma$ is defined by duality.

It is well known that this Signorini problem admits a unique solution
$u\in K := \{ v\in H^1(\Omega) \,:\, v \geq 0\ \text{on}\ \Gamma \}$, see, e.g.,~\cite{GlowinskiLT_81_NAV,Glowinski_08_NMN}.
Moreover, $u$ can be characterized as the unique solution of the variational inequality of the first kind
\begin{align}\label{eq:intro:weakform}
  a(u,v-u) \geq (f,v-u) \quad\text{for all } v\in K
\end{align}
with
\begin{align*}
  a(u,v) := \ip{\nabla u}{\nabla v} + \ip{u}v  \quad\text{for all }u,v\in H^1(\Omega).
\end{align*}
In fact, by choosing appropriate test functions $v\in K$ and integrating by parts, one finds
that problem~\eqref{eq:intro:weakform} is equivalent to~\eqref{eq:intro:signorini} and
\begin{align}\label{eq:intro:bcalt}
  \int_\Gamma \frac{\partial u}{\partial\nn} (v-u) \,d\Gamma \geq 0 \quad\text{for all } v\in K,
\end{align}
see, e.g.,~\cite{GlowinskiLT_81_NAV}.
The last relation is useful to establish a DPG setting of the variational inequality problem. 
Note that~\eqref{eq:intro:bc} implies that $\Gamma$ is partitioned into two parts $\Gamma_D$ and $\Gamma_N$
with $u|_{\Gamma_D} = 0$ and $\partial_\nn u|_{\Gamma_N} = 0$.
Hence, $u$ is the unique solution of~\eqref{eq:intro:signorini} with (mixed) boundary condition
$u|_{\Gamma_D} = 0$, and $\partial_\nn u|_{\Gamma_N} = 0$.
However, $\Gamma_D$ and $\Gamma_N$ are unknown in general and solving \eqref{eq:intro:sigpde} is equivalent
to finding these sets.

\subsection*{Variational formulation}

Let us give a brief overview of our variational setting for the DPG method.
In Section~\ref{sec:dpg} we will consider a non-standard variational form of \eqref{eq:intro:signorini}:
\begin{align}\label{eq:intro:dpg}
  \uu\in U:\quad b(\uu,\vv) = L(\vv) \quad\text{for all } \vv \in V.
\end{align}
Here, $U$ and $V$ are different Hilbert spaces and $b(\cdot,\cdot):\;U\times V\to\R$ is the bilinear form stemming
from, in our case, an ultra-weak formulation. At this point no boundary conditions are included so that there
is no unique solution to \eqref{eq:intro:dpg}.
Denoting by $\ip{\cdot}{\cdot}_V$ the inner product in $V$,
we define the \emph{trial-to-test operator} $\Theta_\beta:\;U\to V$ by
\begin{align}\label{eq:dpg:deftttop}
  \Theta_\beta := \beta\Theta \quad\text{with}\quad
  \ip{\Theta\uu}{\vv}_V := b(\uu,\vv) \quad\text{for all } \vv\in V.
\end{align}
The parameter $\beta>0$ has to be selected. Using this operator, the discretization of \eqref{eq:intro:dpg}
will be based on its equivalent variant with so-called \emph{optimal test functions}:
\begin{align}\label{eq:intro:dpg:theta}
  \uu\in U:\quad b(\uu,\Theta_\beta\vv) = L(\Theta_\beta\vv) \quad\text{for all } \vv \in U.
\end{align}
In our formulations, only one component of $\uu\in U$ corresponds to the original unknown $u$.
Depending on the particular case, we have to define appropriate Dirichlet and Neumann trace operators,
$\gamma_0$ and $\gamma_\nn$, acting on $U$. Then it is left to add
the boundary conditions \eqref{eq:intro:bc} in the form
$\gamma_0 \uu\geq 0$, $\gamma_\nn \uu \geq 0$, and $\gamma_0\uu \gamma_\nn \uu = 0$.
This transforms \eqref{eq:intro:dpg:theta} into a variational inequality.

Keeping in mind~\eqref{eq:intro:bcalt}, we define the bilinear form
\begin{align*}
  a^0(\uu,\vv) := 
  b(\uu,\Theta_\beta \vv) + \dual{\gamma_\nn \uu}{\gamma_0 \vv}_\Gamma \quad\text{for all } \uu,\vv \in U,
\end{align*}
and consider the following formulation:
Find $\uu\in K^0 := \{ \vv\in U \,:\, \gamma_0\vv \geq 0 \}$ such that
\begin{align*}
  a^0(\uu,\vv-\uu) \geq L(\Theta_\beta(\vv-\uu)) \quad\text{for all } \vv\in K^0.
\end{align*}
We will show that this problem is equivalent to~\eqref{eq:intro:sigpde}.
In particular, it has a unique solution.

An intrinsic feature of ultra-weak formulations is that all boundary conditions are essential.
Therefore, we can derive methods that use different convex sets. 
From~\eqref{eq:intro:bc} we infer that
\begin{align*}
  u\left(\frac{\partial v}{\partial \nn}- \frac{\partial u}{\partial \nn}\right) \geq 0
  \quad\text{on }\Gamma\quad \text{for all } v\in H^1(\Omega)
  \text{ with } \frac{\partial v}{\partial \nn}\geq 0 \text{ on }\Gamma,
\end{align*}
giving rise to the formulation: Find $\uu\in K^{\nn} := \set{\vv\in U}{\gamma_\nn\vv\geq 0}$ such that
\begin{align*}
  a^\nn(\uu,\vv-\uu) := b(\uu,\Theta_\beta(\vv-\uu)) + \dual{\gamma_\nn\vv-\gamma_\nn\uu}{\gamma_0\uu}_\Gamma \geq 
  L(\Theta_\beta(\vv-\uu)) \quad\text{for all } \vv\in K^\nn.
\end{align*}
Note that neither $a^0(\cdot,\cdot)$ nor $a^\nn(\cdot,\cdot)$ is symmetric. However,
a combination leads to a formulation with symmetric bilinear form:
Find $\uu\in K^s := \set{\vv\in U}{\gamma_0\vv\geq 0, \gamma_\nn\vv\geq 0}$ such that
\begin{align*}
  a^s(\uu,\vv-\uu) := \tfrac12(a^0(\uu,\vv-\uu)+a^\nn(\uu,\vv-\uu)) \geq L(\Theta_\beta(\vv-\uu)) \quad\text{for all }
  \vv\in K^s.
\end{align*}
For this symmetric case we will establish a simple a posteriori error estimator.

\subsection*{Overview}
The remainder of this paper is structured as follows. Section~\ref{sec:dpg} deals with the
unperturbed model problem \eqref{eq:intro:sigpde} (diffusion parameter $c=1$).
After introducing some notation in \S\ref{sec:dpg:notation},
in \S\ref{sec:dpg:lap} we present a variational inequality that represents \eqref{eq:intro:sigpde}
and is based on an ultra-weak variational formulation. Theorem~\ref{thm:dpg:main} then states
its well-posedness and equivalence. Afterwards, in \S\ref{sec:dpg:disc}, we present the
discrete DPG approximation, prove its well-posedness (Theorem~\ref{thm:dpg:disc}) and
quasi-optimal convergence (Theorem~\ref{thm:dpg:apriori}). An a posteriori error estimate
is derived in \S\ref{sec:aposteriori}. Some technical results are
collected in \S\ref{sec:dpg:tec} and a proof of Theorem~\ref{thm:dpg:main} is given
at the end of Section~\ref{sec:dpg}.

There is a similar structure in Section~\ref{sec:sp} for the singularly perturbed case
(problem \eqref{eq:intro:sigpde} with small diffusion $c=\eps$).
Notation is given in \S\ref{sec:sp:notation} and a variational formulation is presented
and analyzed in \S\ref{sec:sp:lap}. There, Theorem~\ref{thm:sp:main} states the well-posedness
and equivalence of the variational inequality. \S\ref{sec:sp:disc} presents and
analyzes the discrete scheme, its well-posedness (Theorem~\ref{thm:sp:disc}),
quasi-optimality (Theorem~\ref{thm:sp:apriori}) and an a posteriori error estimate
(Theorem~\ref{thm:sp:aposteriori}). Technical results and a proof of Theorem~\ref{thm:sp:main}
are presented in \S\ref{sec:sp:tec} and \S\ref{sec:sp:proof}, respectively.
Finally, in Section~\ref{sec:examples}, we present several numerical examples that underline
our theoretical results for the unperturbed and singularly perturbed cases.

Throughout the paper, $a\lesssim b$ means that $a\le kb$ with a generic constant $k>0$ that is independent of
involved parameters, functions and the underlying mesh. Similarly, we use the notation $a\simeq b$.

\section{DPG method} \label{sec:dpg}

\subsection{Notation}\label{sec:dpg:notation}

For Lipschitz domains $\omega\subset\R^d$ we use the standard Sobolev spaces
$L^2(\omega)$, $H^1(\omega)$, and $\HH(\div,\omega)$.
The $L^2(\omega)$ resp. $L^2(\partial\omega)$ scalar product is denoted by 
$\ip\cdot\cdot_\omega$ resp. $\dual\cdot\cdot_{\partial\omega}$ 
with induced norm $\norm\cdot{\omega}$ resp. $\norm\cdot{\partial\omega}$. 
Also, we define the trace spaces
\begin{align*}
  H^{1/2}(\partial\omega) := \left\{ \trace_\omega u\,:\, u\in H^1(\omega) \right\}
  \quad\text{ and its dual }\quad H^{-1/2}(\partial\omega) := \bigl(H^{1/2}(\partial\omega)\bigr)',
\end{align*}
where $\trace_\omega$ denotes the trace operator. 

Here, duality is understood with respect to $L^2(\partial\omega)$ as a pivot space,
i.e., using the extended $L^2(\partial\omega)$ inner product $\dual{\cdot}{\cdot}_{\partial\omega}$.
The $L^2(\Omega)$ inner product will be denoted by $\ip{\cdot}{\cdot}$ and the corresponding norm by
$\norm\cdot{}$.
Let $\TT$ denote a disjoint partition of $\Omega$ into open Lipschitz sets $\el\in\TT$,
i.e., $\bigcup_{\el\in\TT}\overline\el = \overline\Omega$.
The set of all boundaries of all elements forms the skeleton
$\cS := \left\{ \partial\el \mid \el\in\TT \right\}$.
By $\nn_M$ we mean the outer normal vector on $\partial M$ for a Lipschitz set $M$.
On a partition $\TT$ we use the product spaces 
\begin{align*}
  H^1(\TT) &:= \{ w\in L^2(\Omega) \,:\, w|_T \in H^1(T) \,\forall T\in\TT\}, \\
  \HH(\div,\TT) &:= \{ \qq\in (L^2(\Omega))^d \,:\, \qq|_T \in \HH(\div,T) \,\forall T\in\TT\}.
\end{align*}
The symbols $\pwnabla$, $\pwdiv$, resp. $\pwlap$ denote, the $\TT$-piecewise gradient, divergence, resp. Laplace operators. 
On the skeleton $\cS$ of $\TT$ we introduce the trace spaces
\begin{align*}
  H^{1/2}(\cS) &:=
  \Big\{ \widehat u \in \Pi_{\el\in\TT}H^{1/2}(\partial\el)\,:\,
         \exists w\in H^1(\Omega) \text{ such that } 
         \widehat u|_{\partial\el} = w|_{\partial\el}\; \forall \el\in\TT \Big\},\\
  H^{-1/2}(\cS) &:=
  \Big\{ \widehat\sigma \in \Pi_{\el\in\TT}H^{-1/2}(\partial\el)\,:\,
         \exists \qq\in\HH(\div,\Omega) \text{ such that } 
         \widehat\sigma|_{\partial\el} = (\qq\cdot\nn_{\el})|_{\partial\el}\; \forall\el\in\TT \Big\}.
\end{align*}
These spaces are equipped with norms depending on the problem, see \S\ref{sec:dpg:lap} and \S\ref{sec:sp:notation}.
For functions $\widehat u\in H^{1/2}(\cS)$, $\widehat\sigma\in H^{-1/2}(\cS)$
and $\ttau\in\HH(\div,\TT)$, $v\in H^1(\TT)$ we define
\begin{align*}
  \dual{\widehat u}{\ttau\cdot\nn}_\cS
  := \sum_{\el\in\TT}\dual{\widehat u|_{\partial\el}}{\ttau\cdot\nn_\el}_{\partial\el},\quad
  \dual{\widehat\sigma}{v}_\cS
  := \sum_{\el\in\TT}\dual{\widehat\sigma|_{\partial\el}}{v}_{\partial\el}.
\end{align*}
With the latter relations we can also define the restrictions
$\widehat u|_\Gamma \in H^{1/2}(\Gamma)$ and $\widehat \sigma|_\Gamma \in H^{-1/2}(\Gamma)$
of functions $\widehat u\in H^{1/2}(\cS)$, $\widehat\sigma \in H^{-1/2}(\cS)$ onto $\Gamma$. Let $w\in H^1(\Omega)$ be such that $w|_{\partial T} = \widehat u|_{\partial T}$.
Then, 
\begin{align} \label{trace}
  \dual{\widehat u}{\ttau\cdot\nn}_\cS &= \sum_{T\in\TT}\dual{w|_{\partial T}}{\ttau\cdot\nn_T}_{\partial T} = 
  \sum_{T\in\TT} \ip{\nabla w}{\ttau}_T + \ip{w}{\div \ttau}_T
  = \ip{\nabla w}{\ttau} + \ip{w}{\div \ttau} \nonumber\\
  &= \dual{w|_\Gamma}{\ttau\cdot\nn_\Omega}_\Gamma =: 
  \dual{\widehat u|_\Gamma}{\ttau\cdot\nn_\Omega}_\Gamma
  \quad\text{for all } 
  \ttau\in\HH(\div,\Omega).
\end{align}
Note that $\widehat u|_\Gamma$ is uniquely determined since the above relation is independent of the choice of the
extension $w\in H^1(\Omega)$ and since $H^{-1/2}(\Gamma)$ is the (normal) trace space of $\HH(\div,\Omega)$.
Similarly, we define $\widehat\sigma|_\Gamma \in H^{-1/2}(\Gamma)$ through
\begin{align} \label{ntrace}
  \dual{\widehat \sigma}{v}_\cS &= \sum_{T \in\TT}\dual{\ssigma\cdot\nn_{T}}{v|_\Gamma}_{\partial T} = 
  \sum_{T\in\TT} \ip{\div \ssigma}{v}_T + \ip{\ssigma}{\nabla v}_T
  = \ip{\div \ssigma}{v} + \ip{\ssigma}{\nabla v} \nonumber\\
  &= \dual{\ssigma\cdot\nn_\Omega}{v|_\Gamma}_\Gamma =:
  \dual{\widehat\sigma|_\Gamma}{v|_\Gamma}_\Gamma \quad\text{for all } v\in H^1(\Omega), 
\end{align}
where $\ssigma\in \HH(\div,\Omega)$ with $\ssigma\cdot\nn_{T}|_{\partial T} = \widehat\sigma|_{\partial T}$
for all $T\in\TT$.

\subsection{Ultra-weak variational formulation}\label{sec:dpg:lap}

We derive an ultra-weak formulation of~\eqref{eq:intro:signorini} with $c=1$.
Following~\cite{DemkowiczG_11_ADM}, we define $\ssigma = \nabla u$. Then,
\begin{align*}
  -\div\ssigma + u = f, \quad \ssigma -\nabla u = 0.
\end{align*}
We define $\widehat u$ and $\widehat\sigma$ such that
$\widehat u|_{\partial T} = u|_{\partial T}$ and
$\widehat\sigma|_{\partial T} = \nabla u\cdot\nn_T|_{\partial T}$ for all $T\in\TT$.
Testing the first-order system with functions $v\in H^1(\TT)$, $\ttau\in \HH(\div,\TT)$, and integrating by parts,
we end up with the ultra-weak formulation
\begin{subequations}\label{eq:dpg:varform}
\begin{align}
  \ip{\ssigma}{\pwnabla v} + \ip{u}v - \dual{\widehat\sigma}{v}_{\cS} &= \ip{f}v, \\
  \ip{\ssigma}{\ttau} + \ip{u}{\pwdiv\ttau} - \dual{\widehat u}{\ttau\cdot\nn}_{\cS} &= 0.
\end{align}
\end{subequations}
This gives rise to a bilinear form $b:U \times V \to \R$ and functional $L:V\to \R$ defined by
\begin{align*}
  b(\uu,\vv) &:=  \ip{\ssigma}{\pwnabla v} + \ip{u}v - \dual{\widehat\sigma}{v}_{\cS} 
  + \ip{\ssigma}{\ttau} + \ip{u}{\pwdiv\ttau} - \dual{\widehat u}{\ttau\cdot\nn}_{\cS}, \\
  L(\vv) &:= \ip{f}v
\end{align*}
for all $\uu = (u,\ssigma,\widehat u,\widehat\sigma) \in U$, $\vv = (v,\ttau) \in V$, where
\begin{align*}
  U &:= L^2(\Omega)\times [L^2(\Omega)]^d \times H^{1/2}(\cS) \times H^{-1/2}(\cS), \\
  V &:= H^1(\TT)\times \HH(\div,\TT).
\end{align*}
We equip these spaces with the norms
\begin{align*}
  \norm{\uu}U^2 &:= \norm{u}{}^2 + \norm{\ssigma}{}^2 + \norm{\widehat u}{1/2,\cS}^2 
  +\norm{\widehat\sigma}{-1/2,\cS}^2, \\
  \norm{\vv}V^2 &:= \norm{v}{}^2 + \norm{\pwnabla v}{}^2 + \norm{\ttau}{}^2 + \norm{\pwdiv\ttau}{}^2,
\end{align*}
where the norms for the trace variables are given by the minimum energy extensions to $H^1(\Omega)$ and
$\HH(\div,\Omega)$, respectively, i.e.,
\begin{align*}
  \norm{\widehat u}{1/2,\cS} &:=
  \inf \left\{ (\norm{w}{}^2 + \norm{\nabla w}{}^2)^{1/2} \,:\, w\in H^1(\Omega),
               \widehat u|_{\partial\el}=w|_{\partial\el}\; \forall\el\in\TT \right\},\\
  \norm{\widehat\sigma}{-1/2,\cS} &:=
  \inf \left\{ (\norm{\qq}{}^2\!+\!\norm{\div\qq}{}^2)^{1/2} \,:\, \qq\!\in\!\HH(\div,\Omega), 
               \widehat\sigma|_{\partial\el}\!=\!(\qq\cdot\nn_{\el})|_{\partial\el}\; \forall\el\!\in\!\TT \right\}.
\end{align*}
In $H^{1/2}(\Gamma)$ we use the norm
\begin{align*} 
  \norm{\widehat v}{H^{1/2}(\Gamma)} := 
  \inf \left\{ (\norm{w}{}^2 + \norm{\nabla w}{}^2)^{1/2} \,:\, w\in H^1(\Omega),\widehat v = w|_\Gamma \right\},
\end{align*}
and define the norm $\norm{\cdot}{H^{-1/2}(\Gamma)}$ by duality.

The bilinear form $b(\cdot,\cdot)$ induces a linear operator $B:\;U\to V'$ so that \eqref{eq:dpg:varform} can be
written as
\[
   \uu\in U:\quad B\uu = L.
\]
The operator $B$ has a non-trivial kernel.
In order to consider boundary conditions we define trace operators
\begin{align*}
  \begin{aligned}
    \gamma_0&:\;U\to H^{1/2}(\Gamma),&\quad&
    \gamma_0(u,\ssigma,\widehat u,\widehat\sigma):=\widehat u|_\Gamma
    &\qquad&\text{(Dirichlet trace, see~\eqref{trace}),} \\
    \gamma_\nn&:\;U\to H^{-1/2}(\Gamma),&\quad&
    \gamma_\nn(u,\ssigma,\widehat u,\widehat\sigma):=\widehat \sigma|_\Gamma
    &\qquad&\text{(Neumann trace, see~\eqref{ntrace})}.
  \end{aligned}
\end{align*}
and define the sets
\begin{align*}
  K^0 := \set{\uu\in U}{\gamma_0\uu \geq 0}, \quad K^\nn :=\set{\uu\in U}{\gamma_\nn\uu\geq 0}, 
  \quad K^s := \set{\uu\in U}{\gamma_0\uu\geq 0, \gamma_\nn\uu\geq 0}.
\end{align*}
The relations ``$\geq$'' are partial orderings. Following~\cite[Section~5]{KikuchiO_88_CPE} we define 
\begin{align*}
  \widehat v \geq 0 \text{ in } H^{1/2}(\Gamma) \quad :\Longleftrightarrow \quad
  \exists \{\widehat v_n\}_{n\in \N} \text{ s.t. } v_n \in \mathrm{Lip}(\Gamma) \text{ with }
  v_n\geq 0 \text{ and } v_n \rightharpoonup v \quad\text{in } H^{1/2}(\Gamma),
\end{align*}
where $\mathrm{Lip}(\Gamma)$ denotes all Lipschitz continuous functions $\Gamma \to \R$.
On $H^{-1/2}(\Gamma)$, ``$\geq$'' is understood as duality (see, e.g.,~\cite[Section~1.1.11]{HlavacekHNL_88_SVI}),
\begin{align*}
  \widehat\sigma \geq 0 \text{ in } H^{-1/2}(\Gamma) \quad:\Longleftrightarrow \quad
  \dual{\widehat\sigma}{\widehat v}_\Gamma \geq 0 \quad\text{for all } \widehat v\in H^{1/2}(\Gamma) 
  \text{ with } \widehat v\geq 0.
\end{align*}
One can verify that
$\set{\widehat v\in H^{1/2}(\Gamma)}{\widehat v\geq 0}$ and
$\set{\widehat\sigma\in H^{-1/2}(\Gamma)}{\widehat\sigma\geq 0}$ are closed, convex sets.
We will also see that $K^0$, $K^\nn$, and $K^s$ are non-empty, closed, convex subsets of $U$
(see Lemma~\ref{la_convex} below).

Now let $u\in H^1(\Omega)$ denote the solution of problem~\eqref{eq:intro:sigpde}
and define $\uu=(u,\ssigma,\widehat u,\widehat\sigma)\in U$ with components as above.
Then $B\uu=L$ and, considering inequality~\eqref{eq:intro:bcalt} as a representation of the
boundary conditions~\eqref{eq:intro:bc}, one sees that $\uu\in K^\star$ for $\star\in\{0,\nn,s\}$ and
\begin{align*}
  \dual{\gamma_\nn\uu}{\gamma_0\vv-\gamma_0\uu}_\Gamma &\geq 0 \quad\text{for all }\vv\in K^0, \\
  \dual{\gamma_\nn\vv-\gamma_\nn\uu}{\gamma_0\uu}_\Gamma &\geq 0 \quad\text{for all } \vv\in K^\nn, \\
  \tfrac12(\dual{\gamma_\nn\uu}{\gamma_0\vv-\gamma_0\uu}_\Gamma + 
  \dual{\gamma_\nn\vv-\gamma_\nn\uu}{\gamma_0\uu}_\Gamma) &\geq 0 \quad\text{for all } \vv\in K^s.
\end{align*}
Recalling the formulation \eqref{eq:intro:dpg:theta} of the problem $B\uu=L$,
this leads us to defining the bilinear forms $a^\star:U\times U\to \R$, $\star\in\{0,\nn,s\}$,
and the functional $F:U\to\R$ as follows.
\begin{align*}
  a^0(\uu,\vv) &:= b(\uu,\Theta_\beta\vv) + \dual{\gamma_\nn\uu}{\gamma_0\vv}_\Gamma, \\
  a^\nn(\uu,\vv) &:= b(\uu,\Theta_\beta\vv) + \dual{\gamma_\nn\vv}{\gamma_0\uu}_\Gamma, \\
  a^s(\uu,\vv) &:= b(\uu,\Theta_\beta\vv) 
  + \tfrac12(\dual{\gamma_\nn\uu}{\gamma_0\vv}_\Gamma + \dual{\gamma_\nn\vv}{\gamma_0\uu}_\Gamma), \\
  F(\vv) &:= L(\Theta_\beta\vv)
\end{align*}
for $\uu,\vv\in U$. Here, $\beta>0$ is a constant that will be fixed below.

We then obtain the following formulations of the model problem \eqref{eq:intro:sigpde}
(with $c=1$) as ultra-weak variational inequalities:
For fixed $\star\in\{0,\nn,s\}$ find $\uu\in K^\star$ such that
\begin{align}\label{eq:dpg:varineq}
  a^\star(\uu,\vv-\uu) \geq F(\vv-\uu) \quad\text{for all } \vv\in K^\star.
\end{align}
These are variational inequalities of the first kind and we use a standard framework
for their analysis, cf.~\cite{GlowinskiLT_81_NAV,Glowinski_08_NMN}.

The following is one of our main results.

\begin{theorem}\label{thm:dpg:main}
  Fix $\star\in\{0,\nn,s\}$.
  For all $\beta\geq 2$ there holds the following:
  The bilinear form $a^\star:U\times U \to \R$ is $U$-coercive,
  \begin{align*} 
    \norm{\uu}{U}^2 \leq C_1 a^\star(\uu,\uu) \quad\text{for all } \uu\in U,
  \end{align*}
  and bounded,
  \begin{align*}
    |a^\star(\uu,\vv)| \leq (C_2^2 \beta+1) \norm{\uu}U\norm{\vv}U \quad\text{for all }\uu,\vv \in U.
  \end{align*}
  The constant $C_1>0$ depends only on $\Omega$ and $C_2>0$ is the continuity constant of $b:\;U\times V \to \R$.

  In particular, the variational inequality~\eqref{eq:dpg:varineq} is uniquely solvable and equivalent to
  problem~\eqref{eq:intro:sigpde} with $c=1$ in the following sense:
  If $u\in H^1(\Omega)$ solves problem~\eqref{eq:intro:sigpde},
  then $\uu = (u,\ssigma,\widehat u,\widehat\sigma)\in K^\star$ 
  with $\ssigma := \nabla u$, $\widehat u|_{\partial T} := u|_{\partial T}$,
  $\widehat\sigma|_{\partial T} := \nabla u\cdot\nn_T|_{\partial T}$
  for all $T\in\TT$ solves~\eqref{eq:dpg:varineq}. On the other hand,
  if $\uu = (u,\ssigma,\widehat u,\widehat\sigma)\in K^\star$ solves~\eqref{eq:dpg:varineq}, then $u\in H^1(\Omega)$
  solves~\eqref{eq:intro:sigpde}.

  Moreover, the unique solution $\uu\in K^\star$ of~\eqref{eq:dpg:varineq} satisfies
  \begin{align}\label{eq:dpg:exactsol}
    b(\uu,\Theta_\beta \ww) = F(\ww) \quad\text{for all } \ww\in U.
  \end{align}
\end{theorem}

This theorem is proved in Section~\ref{sec:dpg:proof}.

\subsection{Discretization and convergence}\label{sec:dpg:disc}

To discretize our variational inequality \eqref{eq:dpg:varineq} we use, in this work,
lowest-order piecewise polynomial functions.
That is, we replace the space $U$ by
\begin{align*}
  U_h := P^0(\TT)\times [P^0(\TT)]^d \times S^1(\cS) \times P^0(\cS)
\end{align*}
where $P^0$ denotes the space of element-wise constants on $\TT$ resp. $\cS$, and $S^1(\cS)$ is
the space of globally continuous, element-wise affine functions on $\cS$.
Defining the non-empty convex subsets
\begin{align*}
  K_h^0 &:= \set{\vv_h\in U_h}{\gamma_0\vv_h\geq 0}, \\
  K_h^\nn &:= \set{\vv_h\in U_h}{\gamma_\nn\vv_h\geq 0}, \\
  K_h^s &:= \set{\vv_h\in U_h}{\gamma_0\vv_h\geq 0, \gamma_\nn\vv_h\geq 0}
\end{align*}
we find that $K_h^\star\subseteq K^\star$. 

The discretized version of~\eqref{eq:dpg:varineq} then reads: For fixed $\star\in\{0,\nn,s\}$ 
find $\uu_h\in K_h^\star$ such that
\begin{align}\label{eq:dpg:varineqdisc}
  a^\star(\uu_h,\vv_h-\uu_h)\geq F(\vv_h-\uu_h) \quad\text{for all } \vv_h\in K_h^\star.
\end{align}
Coercivity and boundedness of $a^\star(\cdot,\cdot)$ hold on the full space $U$ (Theorem~\ref{thm:dpg:main}).
Therefore, the Lions-Stampacchia theorem applies also in the discrete case.

\begin{theorem} \label{thm:dpg:disc}
Under the assumptions of Theorem~\ref{thm:dpg:main}, 
the discrete variational inequality~\eqref{eq:dpg:varineqdisc} admits a unique solution $\uu_h\in K_h^\star$.
\end{theorem}

Our scheme converges quasi-optimally in the following sense.

\begin{theorem}\label{thm:dpg:apriori}
  For $\beta\ge 2$ let $\uu\in K^\star$, $\uu_h\in K_h^\star$ denote the exact solutions
  of~\eqref{eq:dpg:varineq}, \eqref{eq:dpg:varineqdisc}. Then there holds
  \begin{align*} 
    \norm{\uu-\uu_h}U^2 \lesssim 
    \begin{cases}
      \inf_{\vv_h\in K_h^0} \left( \norm{\uu-\vv_h}U^2 +\dual{\gamma_\nn\uu}{\gamma_0(\vv_h-\uu)}_\Gamma\right)
      & (\star = 0), \\
      \inf_{\vv_h\in K_h^\nn} \left( \norm{\uu-\vv_h}U^2 
      +\dual{\gamma_\nn(\vv_h-\uu)}{\gamma_0\uu}_\Gamma\right)  & (\star = \nn), \\
      \inf_{\vv_h\in K_h^s} \left( \norm{\uu-\vv_h}U^2 
      + \tfrac12(\dual{\gamma_\nn\uu}{\gamma_0(\vv_h-\uu)}_\Gamma 
      + \dual{\gamma_\nn(\vv_h-\uu)}{\gamma_0\uu}_\Gamma)\right)
      & (\star = s).
    \end{cases}
  \end{align*}
  The generic constant depends on $\Omega$ and $\beta$ but not on $\TT$.
\end{theorem}

\begin{proof}
By Theorem~\ref{thm:dpg:main}, $a^\star(\cdot,\cdot)$ is $U$-coercive and bounded. Therefore we can
follow Falk's lemma~\cite{Falk_74_EEA} to deduce that
\begin{align*}
  \norm{\uu-\uu_h}{U}^2 &\lesssim a^\star(\uu-\uu_h,\uu-\uu_h) = a^\star(\uu-\uu_h,\uu-\vv_h) 
  + a^\star(\uu-\uu_h,\vv_h-\uu_h) \\
  & \leq \norm{\uu-\uu_h}U (C_2^2\beta+1) \norm{\uu-\vv_h}U + a^\star(\uu-\uu_h,\vv_h-\uu_h)
\end{align*}
for all $\vv_h\in K_h^\star\subseteq K^\star$.
We consider only the case $\star = 0$. The remaining cases are treated in the same manner.
To tackle the last term on the right-hand side we use~\eqref{eq:dpg:exactsol} and~\eqref{eq:dpg:varineqdisc} to see that
\begin{align*}
  a^0(\uu-\uu_h,\vv_h-\uu_h) &= b(\uu,\Theta_\beta (\vv_h-\uu_h) ) + \dual{\gamma_\nn\uu}{\gamma_0(\vv_h-\uu_h)}_\Gamma 
  - a^0(\uu_h,\vv_h-\uu_h) \\
  &= \dual{\gamma_\nn\uu}{\gamma_0(\vv_h-\uu_h)}_\Gamma + F(\vv_h-\uu_h)-a^0(\uu_h,\vv_h-\uu_h) \\
  & \leq \dual{\gamma_\nn\uu}{\gamma_0(\vv_h-\uu_h)}_\Gamma.
\end{align*}
Note that the exact solution $\uu$ satisfies $\gamma_\nn \uu \geq 0$ and $\gamma_0\uu \gamma_\nn\uu = 0$. For the
discrete one there holds $\gamma_0\uu_h\geq 0$.
Hence, the last term further simplifies to
\begin{align*}
  \dual{\gamma_\nn\uu}{\gamma_0(\vv_h-\uu_h)}_\Gamma \leq \dual{\gamma_\nn\uu}{\gamma_0 \vv_h}_\Gamma
  = \dual{\gamma_\nn\uu}{\gamma_0(\vv_h-\uu)}_\Gamma.
\end{align*}
Altogether, Young's inequality with some parameter $\delta>0$ shows that
\begin{align*}
  \norm{\uu-\uu_h}{U}^2\lesssim \frac{\delta}2 \norm{\uu-\uu_h}{U}^2 
  + \frac{\delta^{-1}}2 (C_2^2\beta+1)^2\norm{\uu-\vv_h}{U}^2 
  + \dual{\gamma_\nn\uu}{\gamma_0(\vv_h-\uu)}_\Gamma
\end{align*}
for arbitrary $\vv_h\in K_h^0$. This proves the a priori estimate.
\end{proof}

\begin{remark}\label{rem:approx}
  To deduce convergence rates assume, for instance, that $u\in H^3(\Omega)$ is the solution
      of~\eqref{eq:intro:sigpde}. Set
  \begin{align*}
    \vv_h := (\Pi_h^0 u,\boldsymbol{\Pi}_h^0\nabla u,I_h u|_\cS, (\Pi_h^{\div} \nabla u\cdot\nn_T|_{\partial T})_{T\in\TT})
    \in U,
  \end{align*}
  where $\Pi_h^0$ and $\boldsymbol{\Pi}_h^0$ are the $L^2(\Omega)$-orthogonal projections onto the element-wise
  constant spaces, $I_h$ is the nodal interpolant, and $\Pi_h^{\div}$ is the (lowest-order) Raviart-Thomas projector.
  Note that $I_h$ preserves non-negativity (in particular, on the boundary) and
  the normal trace of $\Pi_h^{\div}$ is the $L^2(\Gamma)$-orthogonal projection $\pi_h^0$ and, hence, preserves
  non-negativity on the boundary as well. 
  Thus, $\vv_h \in K_h^\star$ for $\star\in\{0,\nn,s\}$. We refer to~\cite{DemkowiczG_11_ADM} to see that
  $\norm{\uu-\vv_h}U = O(h)$.
  To control the boundary terms we note that
  \begin{align*}
    |\dual{\gamma_\nn\uu}{\gamma_0\vv_h-\gamma_0\uu}_\Gamma| \leq \norm{\gamma_\nn\uu}{\Gamma}
    \norm{(u-I_h u)|_\Gamma}{\Gamma} = \OO(h^2).
  \end{align*}
  Let $\Gamma_1,\dots,\Gamma_L \subseteq \Gamma$ be such that $\bigcup \overline\Gamma_j = \Gamma$ 
  and $\nn_\Omega|_{\Gamma_j}$ is constant. Using the projection
  property we obtain
  \begin{align*}
    &|\dual{\gamma_\nn\vv_h-\gamma_\nn\uu}{\gamma_0\uu}_\Gamma| \leq 
    \sum_{E\in\cS_\Gamma} |\dual{\gamma_\nn\vv_h-\gamma_\nn\uu}{\gamma_0\uu}_E|
    = \sum_{E\in\cS_\Gamma} |\dual{\gamma_\nn\vv_h-\gamma_\nn\uu}{\gamma_0\uu-\pi_h^0\gamma_0\uu}_E| \\
    &\qquad \lesssim \sum_{E\in\cS_\Gamma} |E| \norm{\gamma_\nn\uu}{H^1(E)} |E| \norm{\gamma_0\uu}{H^1(E)}
    \lesssim h^2 \left(\sum_{j=1}^L \norm{\gamma_\nn\uu}{H^1(\Gamma_j)}^2\right)^{1/2} \norm{\gamma_0\uu}{H^1(\Gamma)}.
  \end{align*}
  Altogether, using $\vv_h$ in Theorem~\ref{thm:dpg:apriori} we infer
  \begin{align*}
    \norm{\uu-\uu_h}U^2 &\lesssim \norm{\uu-\vv_h}U^2 + 
      \begin{cases}
      |\dual{\gamma_\nn\uu}{\gamma_0(\vv_h-\uu)}_\Gamma|
      & (\star = 0) \\
      |\dual{\gamma_\nn(\vv_h-\uu)}{\gamma_0\uu}_\Gamma|  & (\star = \nn) \\
      \tfrac12 |\dual{\gamma_\nn\uu}{\gamma_0(\vv_h-\uu)}_\Gamma 
      + \dual{\gamma_\nn(\vv_h-\uu)}{\gamma_0\uu}_\Gamma|
      & (\star = s)
    \end{cases}
    \\&= O(h^2).
  \end{align*}
    With less regularity of $u$ the treatment of the boundary terms becomes more technical. We refer
    to~\cite{DrouetH_15_OCD} for details.
\end{remark}

\begin{remark}
  Our analysis also allows for non-conforming discrete cones $K_h^\star\not\subseteq K^\star$.
  Then, an additional consistency error shows up in Theorem~\ref{thm:dpg:apriori}, that is, in the case $\star = 0$,
  \begin{align*}
    \norm{\uu-\uu_h}U^2 \lesssim \inf_{\vv_h\in K_h^0} \left( \norm{\uu-\vv_h}{U}^2 +
    \dual{\gamma_\nn\uu}{\gamma_0\vv_h-\gamma_0\uu}_\Gamma \right) + \inf_{\vv\in K^0}
    \dual{\gamma_\nn\uu}{\gamma_0\vv-\gamma_0\uu_h}_\Gamma.
  \end{align*}
  Again, this a priori estimate can be derived by using Falk's lemma.
\end{remark}

\subsection{A posteriori error estimate} \label{sec:aposteriori}

We derive a simple error estimator in the symmetric case, i.e., $\star=s$.
Throughout this section we assume that $\beta>0$ is a fixed constant such that $a^s(\cdot,\cdot)$ is coercive
(Theorem~\ref{thm:dpg:main}).

Let $\cS_\Gamma := \set{\partial T \cap \Gamma}{T\in\TT}$ denote the mesh on the boundary which is
induced by the volume mesh $\TT$. Furthermore let $\uu_h\in K_h^s$ denote the unique solution of~\eqref{eq:dpg:varineqdisc}.
We define for all $T\in\TT$ local volume error indicators by
\begin{align} \label{eq:aposteriori:estdef}
  \eta(T)^2 := \beta \norm{R_T^{-1}\iota_T^*(L-B\uu_h)}{V(T)}^2.
\end{align}
Here, $V(T) := H^1(T) \times \HH(\div,T)$, $\norm{\cdot}{V(T)}$ denotes the canonical norm on $V(T)$,
$R_T:\;V(T) \to (V(T))'$ is the Riesz isomorphism,
and $\iota_T^*$ is the dual of the canonical embedding $\iota_T:\; V(T)\to V$.
Moreover, for all $E\in\cS_\Gamma$ we define local boundary indicators by
\begin{align*}
  \eta(E)^2 := \dual{\gamma_\nn\uu_h}{\gamma_0\uu_h}_E.
\end{align*}
Note that $\uu_h\in K_h^s$ implies that $\eta(E)^2\geq 0$. 
The overall estimator is then given by
\begin{align*}
  \eta^2 := \sum_{T\in\TT} \eta(T)^2 + \sum_{E\in\cS_\Gamma} \eta(E)^2.
\end{align*}

\begin{theorem}\label{thm:dpg:aposteriori}
  For $\beta\ge 2$
  let $\uu\in K^s$ and $\uu_h\in K_h^s$ be the solutions of~\eqref{eq:dpg:varineq} and~\eqref{eq:dpg:varineqdisc},
  respectively. Then, there holds the reliability estimate
  \begin{align*} 
    \norm{\uu-\uu_h}U \leq C_\mathrm{rel} \eta
  \end{align*}
  with a constant $C_\mathrm{rel}>0$ that depends on $\Omega$ but not on $\TT$ or $\beta$.
\end{theorem}

\begin{proof}
  By the $U$-coercivity of $a^s(\cdot,\cdot)$ (see Theorem~\ref{thm:dpg:main}) we have
  \begin{align*}
    \norm{\uu-\uu_h}U^2 &\lesssim a^s(\uu-\uu_h,\uu-\uu_h) 
    = b(\uu-\uu_h,\Theta_\beta(\uu-\uu_h) ) + \dual{\gamma_\nn(\uu-\uu_h)}{\gamma_0(\uu-\uu_h)}_\Gamma
    \\ &=\beta \norm{B(\uu-\uu_h)}{V'}^2 + \dual{\gamma_\nn(\uu-\uu_h)}{\gamma_0(\uu-\uu_h)}_\Gamma.
  \end{align*}
  Note that $\uu\in K^s$ and $\uu_h\in K_h^s$ implies that $\dual{\gamma_\nn\uu}{\gamma_0\uu_h}_\Gamma\geq 0$ and
  $\dual{\gamma_\nn\uu_h}{\gamma_0\uu}_\Gamma\geq 0$.
  Together with $B\uu = L$ and $\gamma_0\uu\gamma_\nn\uu = 0$ we obtain the estimate
  \begin{align*}
    \norm{\uu-\uu_h}U^2 \lesssim 
    \beta \norm{L-B\uu_h}{V'}^2 + \dual{\gamma_\nn \uu_h}{\gamma_0\uu_h}_\Gamma
    = \sum_{T\in\TT} \eta(T)^2 + \sum_{E\in\cS_\Gamma} \eta(E)^2,
  \end{align*}
  which finishes the proof.
\end{proof}

\subsection{Technical details} \label{sec:dpg:tec}

We start by proving boundedness of our trace operators which are specifically modified for the space $U$.

\begin{lemma}\label{lem:trace}
  $\gamma_0 : (U,\norm\cdot{U})\to (H^{1/2}(\Gamma),\norm\cdot{H^{1/2}(\Gamma)})$ and 
  $\gamma_\nn : (U,\norm\cdot{U})\to (H^{-1/2}(\Gamma),\norm\cdot{H^{-1/2}(\Gamma)})$ are bounded with constant $1$.
\end{lemma}

\begin{proof}
Boundedness of these operators follows basically by definition of the corresponding norms,
see also~\cite[Lemma~3]{FuehrerHK_CDB}. Specifically, the definitions of the norms
$\norm\cdot{H^{1/2}(\Gamma)}$, $\norm\cdot{1/2,\cS}$, and $\norm\cdot{U}$ prove boundedness of $\gamma_0$.
Note that $\norm{\cdot|_\Gamma}{H^{-1/2}(\Gamma)} \leq \norm{\cdot}{-1/2,\cS}$ and, thus, the definition of
$\norm\cdot{U}$ shows boundedness of $\gamma_\nn$.
\end{proof}

By the boundedness of the operators $\gamma_0$ and $\gamma_\nn$ we immediately establish the following result.

\begin{lemma} \label{la_convex}
  The sets $K^\star$ ($\star\in\{0,\nn,s\}$) are non-empty, closed, convex subsets of $U$.
\end{lemma}

The following steps are to characterize the kernel of the operator $B$. This kernel is non-trivial
since $B$ does not include any boundary condition. Our procedure is similar to the one in~\cite{FuehrerHK_CDB}.

For a function $v\in H^{1/2}(\Gamma)$, we define its quasi-harmonic extension $\widetilde u\in H^1(\Omega)$ as the
unique solution of
\begin{align}\label{eq:dpg:harmext}
  -\Delta \widetilde u + \widetilde u &=0 \quad\text{in }\Omega, \qquad
  \widetilde u|_\Gamma = v.
\end{align}
Note that the infimum in the definition of $\norm{v}{H^{1/2}(\Gamma)}$ is attained for the function $\widetilde u$,
i.e., $\norm{\widetilde u}{}^2 + \norm{\nabla\widetilde u}{}^2 = \norm{v}{H^{1/2}(\Gamma)}^2$.
Then, define the operator $\ext : H^{1/2}(\Gamma)\to U$ by
\begin{align*}
  \ext v := (\widetilde u,\ssigma,\widehat u,\widehat\sigma), \text{ where }
  \ssigma := \nabla\widetilde u, \quad \widehat u|_{\partial T} := \widetilde u|_{\partial T}, 
  \quad \widehat\sigma|_{\partial T} := \nabla\widetilde u\cdot\nn_T|_{\partial T} \quad\text{for all } T\in\TT.
\end{align*}
The range of this operator is the kernel of $B$.
We combine important properties of $B$ and $\ext$ in the following lemma.

\begin{lemma}\label{lem:kerBext}
  The operators $B : U \to V'$, $\ext: H^{1/2}(\Gamma)\to U$ have the following properties:
  \begin{enumerate}[(i)]
    \item\label{lem:kerBext:a} $B$ and $\ext$ are bounded. Specifically there holds
      \begin{align*}
        &\norm{B\uu}{V'} \lesssim \norm{\uu}{U} &&\hspace{-7em}\text{for all }\uu\in U,\\
        &\norm{v}{H^{1/2}(\Gamma)} \le \norm{\ext v}{U} \le \sqrt{3} \norm{v}{H^{1/2}(\Gamma)}
                                                &&\hspace{-7em}\text{for all }v \in H^{1/2}(\Gamma).
      \end{align*}
    \item\label{lem:kerBext:b} The kernel of $B$ consists of all quasi-harmonic extensions, i.e., $\ker(B) = \ran(\ext)$.
    \item\label{lem:kerBext:c} $\ext$ is a right-inverse of $\gamma_0$.
    \item\label{lem:kerBext:d} $B:\;U/\ker(B)\to V'$ is inf-sup stable,
      \begin{align*}
        \norm{\uu-\ext\gamma_0\uu}{U} \lesssim \norm{B\uu}{V'} = \sup_{\vv\in V} \frac{\dual{B\uu}{\vv}}{\norm{\vv}{V}}
        = b(\uu,\Theta\uu)^{1/2} \quad\text{for all } \uu\in U.
      \end{align*}
      The involved constant depends only on $\Omega$.
  \end{enumerate}
\end{lemma}
\begin{proof}
In the case of the Poisson equation these results have been established in~\cite{FuehrerHK_CDB}.
In a recent work~\cite{FuehrerHS_TSD} we analyzed a time-stepping scheme for the heat equation which naturally leads
to the equation $-\Delta u + \delta^{-1}u = f$ where $\delta$ corresponds to a time step $k_n$.
Setting $\delta=k_n=1$ in \cite[Lemma 8]{FuehrerHS_TSD} we obtain boundedness of the operator $B$ and stability
\begin{align*}
  \norm{\uu_0}U \lesssim \norm{B\uu_0}{V'} \quad\text{for all } \uu_0 \in U \text{ with } \gamma_0\uu_0 = 0.
\end{align*}
By definition of $\ext$ one sees~\eqref{lem:kerBext:c}. 
Furthermore, integration by parts shows $\ran(\ext)\subseteq \ker(B)$.
For $\uu\in U$ we define $\uu_0:=\uu-\ext\gamma_0\uu$ and infer
\begin{align*}
  \norm{\uu-\ext\gamma_0\uu}U \lesssim \norm{B(\uu-\ext\gamma_0\uu)}{V'} = \norm{B\uu}{V'},
\end{align*}
which proves~\eqref{lem:kerBext:d} as well as $\ker(B)\subseteq \ran(\ext)$, hence,~\eqref{lem:kerBext:b}.

It remains to show the relations for $\ext$ in~\eqref{lem:kerBext:a}.
Let $\uu = (u,\ssigma,\widehat u,\widehat\sigma) = \ext v$ for $v\in H^{1/2}(\Gamma)$. 
Then,
\begin{align*}
  \norm{\ext v}U^2 \geq \norm{u}{}^2 + \norm{\ssigma}{}^2 = \norm{u}{}^2 + \norm{\nabla u}{}^2 =
  \norm{v}{H^{1/2}(\Gamma)}^2.
\end{align*}
On the other hand, using that $\Delta u = u$ by construction of $\ext v$, we deduce that
\begin{align*}
  \norm{\ext v}U^2 &= \norm{u}{}^2 + \norm{\ssigma}{}^2  + \norm{\widehat u}{1/2,\cS}^2 
  + \norm{\widehat\sigma}{-1/2,\Gamma}^2 \\
  &\leq \norm{u}{}^2 + \norm{\nabla u}{}^2 +\norm{u}{}^2 + \norm{\nabla u}{}^2
  + \norm{\nabla u}{}^2 + \norm{\Delta u}{}^2 \\
  &= 3\norm{u}{}^2 + 3\norm{\nabla u}{}^2 = 3\norm{v}{H^{1/2}(\Gamma)}^2.
\end{align*}
This concludes the proof.
\end{proof}

In Lemma~\ref{lem:dpg:Hm12estimate} below we give an explicit bound of the control of the Neumann trace
for functions of the quotient space $U/\ker(B)$.
For its proof we need the following technical result.

\begin{lemma}\label{lem:dpg:dualestimate}
  Let $\widehat v\in H^{1/2}(\Gamma)$. The problem
  \begin{subequations}\label{eq:dpg:dualestimate}
  \begin{align}
    \div\ttau + v &= 0\quad\text{in } \Omega, \label{eq:dpg:dualestimate:a} \\
    \ttau + \nabla v &= 0\quad\text{in } \Omega, \label{eq:dpg:dualestimate:b} \\
    v|_{\Gamma} &= \widehat v \label{eq:dpg:dualestimate:c}
  \end{align}
  \end{subequations}
  admits a unique solution $(v,\ttau) \in H^1(\Omega)\times \HH(\div,\Omega)$ with $\Delta v\in H^1(\Omega)$ and
  \begin{align*}
    \norm{\ttau}{}^2 + \norm{\div\ttau}{}^2 = \norm{\nabla v}{}^2 + \norm{v}{}^2 = 
    \norm{\widehat v}{H^{1/2}(\Gamma)}^2.
  \end{align*}
\end{lemma}
\begin{proof}
Let $v\in H^1(\Omega)$ be the unique solution of 
\begin{align*}
  -\Delta v + v = 0\quad\text{in } \Omega, \qquad v|_\Gamma = \widehat v.
\end{align*}
Then, $\norm{\nabla v}{}^2 + \norm{v}{}^2 = \norm{\widehat v}{H^{1/2}(\Gamma)}^2$ by definition of the latter norm.
Define $\ttau:=-\nabla v\in L^2(\Omega)$. Since $\Delta v = v \in H^1(\Omega)$, we have $\ttau\in \HH(\div,\Omega)$,
and $\div\ttau = -\Delta v = -v$ shows~\eqref{eq:dpg:dualestimate:a}.
To see unique solvability, let additionally $(v_2,\ttau_2)\in H^1(\Omega)\times \HH(\div,\Omega)$
solve~\eqref{eq:dpg:dualestimate}. The difference $w:=v-v_2$ satisfies $-\Delta w + w =0$ in $\Omega$
with $w|_\Gamma = 0$. Thus $w=0$ and $\ttau = -\nabla v = -\nabla v_2 = \ttau_2$ as well.
\end{proof}

\begin{lemma}\label{lem:dpg:Hm12estimate}
  There holds
  \begin{align}\label{eq:dpg:Hm12estimate}
    \norm{\gamma_\nn(\uu-\ext\gamma_0\uu)}{H^{-1/2}(\Gamma)} \leq \sqrt{2}\norm{B\uu}{V'} \quad\text{for all } \uu\in U.
  \end{align}
\end{lemma}
\begin{proof}
  Let $\widehat v\in H^{1/2}(\Gamma)$ and choose the test function $\vv=(v,\ttau)\in H^1(\Omega) \times
  \HH(\div,\Omega) \subseteq V$ to be the solution of~\eqref{eq:dpg:dualestimate}.
  By Lemma~\ref{lem:dpg:dualestimate} there holds $\norm{\vv}{V}=\sqrt{2}\norm{\widehat v}{H^{1/2}(\Gamma)}$.
  Then, by the definition of the bilinear form $b(\cdot,\cdot)$, and since $\gamma_0(\uu-\ext\gamma_0\uu) = 0$, we have
  \begin{align*}
    \lvert \dual{\gamma_\nn(\uu-\ext\gamma_0\uu)}{\widehat v}_\Gamma \rvert 
    &= \lvert b(\uu-\ext\gamma_0\uu,\vv)\rvert = \lvert b(\uu,\vv)\rvert
    \leq \norm{B\uu}{V'} \norm{\vv}{V}
    = \sqrt{2} \norm{B\uu}{V'} \norm{\widehat v}{H^{1/2}(\Gamma)},
  \end{align*}  
  where we have used that $\ext\gamma_0\uu\in \ker(B)$. 
  Dividing by $\norm{\widehat v}{H^{1/2}(\Gamma)}$ and taking the supremum over all $\widehat v\in H^{1/2}(\Gamma) \setminus 
  \{0\}$, this proves~\eqref{eq:dpg:Hm12estimate}.
\end{proof}

\subsection{Proof of Theorem~\ref{thm:dpg:main}} \label{sec:dpg:proof}

First, we prove boundedness and coercivity of $a^\star(\cdot,\cdot)$.
Then, the Lions-Stampacchia theorem, see, e.g.,~\cite{Glowinski_08_NMN,GlowinskiLT_81_NAV,Rodrigues_87_OPM},
proves unique solvability of~\eqref{eq:dpg:varineq} provided that $F : U\to \R$ is a
linear functional, which follows by the boundedness of $\Theta:U\to\R$.

We show the boundedness. Note that $b:U\times V\to \R$ and $\Theta:U\to V$ are uniformly bounded,
\begin{align*}
  |b(\uu,\Theta_\beta\vv)| \leq \beta C_2 \norm{\uu}U\norm{\Theta\vv}V \leq \beta C_2^2\norm{\uu}U\norm{\vv}U
  \quad\text{for all }\uu,\vv\in U.
\end{align*}
By duality and boundedness of $\gamma_0,\gamma_\nn$ (see Lemma~\ref{lem:trace}) we have
\begin{align*}
  |\dual{\gamma_\nn\uu}{\gamma_0\vv}_\Gamma| \leq \norm{\gamma_\nn\uu}{H^{-1/2}(\Gamma)}
  \norm{\gamma_0\vv}{H^{1/2}(\Gamma)} \leq \norm{\uu}{U}\norm{\vv}U \quad\text{for all }\uu,\vv\in U.
\end{align*}
Next, we follow~\cite{FuehrerHK_CDB} to prove coercivity.
Note that $a^0(\uu,\uu) = a^\nn(\uu,\uu) = a^s(\uu,\uu)$ for all $\uu\in U$.
Let $\uu\in U$. By the triangle inequality and 
Lemma~\ref{lem:kerBext} we get
\begin{align*}
  \norm{\uu}{U}^2 \lesssim \norm{\uu-\ext\gamma_0\uu}U^2 + \norm{\ext\gamma_0\uu}U^2 \lesssim \norm{B\uu}{V'}^2 + 
  \norm{\gamma_0\uu}{H^{1/2}(\Gamma)}^2.
\end{align*}
Let $\widetilde u\in H^1(\Omega)$ be the quasi-harmonic extension~\eqref{eq:dpg:harmext} of $\gamma_0\uu$.
The definition of the $H^{1/2}(\Gamma)$-norm and integration by parts show that
\begin{align*}
  \norm{\gamma_0\uu}{H^{1/2}(\Gamma)}^2 &= \norm{\widetilde u}{H^1(\Omega)}^2 = \norm{\widetilde u}{L^2(\Omega)}^2
  + \norm{\nabla \widetilde u}{L^2(\Omega)}^2
  = \dual{\partial_\nn \widetilde u}{\widetilde u|_\Gamma}_\Gamma = \dual{\gamma_\nn \ext\gamma_0\uu}{\gamma_0\uu}_\Gamma.
\end{align*}
This gives
\begin{align*}
  \norm{B\uu}{V'}^2 + \dual{\gamma_\nn\ext\gamma_0\uu}{\gamma_0\uu}_\Gamma 
  = \norm{B\uu}{V'}^2 + \dual{\gamma_\nn\ext\gamma_0\uu-\gamma_\nn\uu}{\gamma_0\uu}_\Gamma + 
  \dual{\gamma_\nn\uu}{\gamma_0\uu}_\Gamma.
\end{align*}
Using duality, Lemma~\ref{lem:dpg:Hm12estimate}, and Young's inequality with some parameter $\delta>0$, we obtain
\begin{align*}
  \dual{\gamma_\nn\ext\gamma_0\uu-\gamma_\nn\uu}{\gamma_0\uu}_\Gamma &\leq
  \norm{\gamma_\nn\ext\gamma_0\uu-\gamma_\nn\uu}{H^{-1/2}(\Gamma)} \norm{\gamma_0\uu}{H^{1/2}(\Gamma)}
  \leq \sqrt{2}\norm{B\uu}{V'}\norm{\gamma_0\uu}{H^{1/2}(\Gamma)} \\
  &\leq \delta^{-1} \norm{B\uu}{V'}^2 + \frac{\delta}2 \norm{\gamma_0\uu}{H^{1/2}(\Gamma)}^2.
\end{align*}
For $\delta = 1$ we get
\begin{align*}
  \norm{B\uu}{V'}^2 + \tfrac12 \norm{\gamma_0\uu}{H^{1/2}(\Gamma)}^2 \leq 2 \norm{B\uu}{V'}^2 +
  \dual{\gamma_\nn\uu}{\gamma_0\uu}_\Gamma.
\end{align*}
Thus,
\begin{align*}
  C_1^{-1} \norm{\uu}U^2 \leq \beta\norm{B\uu}{V'}^2 +\dual{\gamma_\nn\uu}{\gamma_0\uu}_\Gamma = a^\star(\uu,\uu)
\end{align*}
for all $\beta\geq 2$, where the constant $C_1>0$ depends only on $\Omega$.

Regarding the equivalence of problems~\eqref{eq:intro:sigpde} and~\eqref{eq:dpg:varineq}
we know that the unique solution $u$ of~\eqref{eq:intro:sigpde} with $\uu$
defined as in the assertion satisfies~\eqref{eq:dpg:varineq} by construction. 
The other direction follows by existence of a unique solution of~\eqref{eq:dpg:varineq}.

Finally,
note that~\eqref{eq:dpg:exactsol} follows by construction. However, in the case $\star=0$ we can infer this identity
directly from~\eqref{eq:dpg:varineq}: 
Let $\uu$ denote the solution of~\eqref{eq:dpg:varineq} with $\star=0$.
Set $\vv = \uu\pm (\ww-\ext\gamma_0\ww)$ for some arbitrary $\ww\in U$.
Since $\gamma_0(\ww-\ext\gamma_0\ww) = 0$, we infer $\vv\in K^0$, so that we can use it as a test
function in~\eqref{eq:dpg:varineq}. This gives
\begin{align*}
  \pm a^0(\uu,\ww-\ext\gamma_0\ww) \geq \pm F(\ww-\ext\gamma_0\ww).
\end{align*}
Hence, 
\begin{align*}
  a^0(\uu,\ww-\ext\gamma_0\ww) = F(\ww-\ext\gamma_0\ww) \quad\text{for all }\ww\in U.
\end{align*}
Note that $\ext\gamma_0\ww\in\ker B = \ker \Theta$. This leads to 
\begin{align*}
  F(\ww-\ext\gamma_0\ww) &= L(\Theta_\beta (\ww-\ext\gamma_0\ww)) = L(\Theta_\beta\ww) = F(\ww), \text{ and }\\
  a^0(\uu,\ww-\ext\gamma_0\ww) &= b(\uu,\Theta_\beta\ww) + \dual{\gamma_\nn\uu}{\gamma_0(\ww-\ext\gamma_0\ww)}_\Gamma 
  = b(\uu,\Theta_\beta\ww).
\end{align*}
This concludes the proof.

\section{DPG method for a singularly perturbed problem}\label{sec:sp}

In this section we introduce and analyze a DPG method for the Signorini problem of
reaction dominated diffusion, that is, for \eqref{eq:intro:sigpde} with small positive constant
$c=\eps$ ($0<\eps\le 1$).

We mainly stick to the notation as given in Section~\ref{sec:dpg}, but need some additional definitions.
Also, we redefine some objects like the bilinear form $b:\;U\times V\to\R$,
the spaces $U$, $V$, and some norms. In fact, our objective of robust control of field variables
forces us to carefully scale parts of norms with coefficients depending on the diffusion parameter $\eps$.

\subsection{Notation} \label{sec:sp:notation}

For functions $\widehat u \in H^{1/2}(\cS)$, $\widehat\sigma\in H^{-1/2}(\cS)$ we define the skeleton norms
\begin{align*}
  \norm{\widehat u}{1/2,\cS} &:=
  \inf \left\{ (\norm{w}{}^2 + \eps^{1/2}\norm{\nabla w}{}^2)^{1/2} \,:\, w\in H^1(\Omega),
               \widehat u|_{\partial\el}=w|_{\partial\el}\; \forall\el\in\TT \right\},\\
  \norm{\widehat\sigma}{-1/2,\cS} &:=
  \inf \left\{ (\norm{\qq}{}^2\!+\!\eps\norm{\div\qq}{}^2)^{1/2} \,:\, \qq\!\in\!\HH(\div,\Omega), 
               \widehat\sigma|_{\partial\el}\!=\!(\qq\cdot\nn_{\el})|_{\partial\el}\; \forall\el\!\in\!\TT \right\}.
\end{align*}
Moreover, for $\widehat u\in H^{1/2}(\Gamma)$, $\widehat\sigma\in H^{1/2}(\Gamma)$
we define the boundary norms
\begin{align*}
  \norm{\widehat u}{1/2,\Gamma} &:=
  \inf \left\{ (\norm{w}{}^2 + \eps^{1/2}\norm{\nabla w}{}^2)^{1/2} \,:\, w\in H^1(\Omega),
               w|_\Gamma=\widehat u \right\},\\
  \norm{\widehat\sigma}{-1/2,\Gamma} &:=
  \inf \left\{ (\norm{\qq}{}^2\!+\!\eps\norm{\div\qq}{}^2)^{1/2} \,:\, \qq\!\in\!\HH(\div,\Omega), 
               (\qq\cdot\nn_{\Omega})|_{\Gamma}=\widehat\sigma \right\}.
\end{align*}
We will need another norm in $H^{1/2}(\Gamma)$ defined by
\begin{align*}
  \norm{\widehat u}{H^{1/2}(\Gamma)}
  :=
  \inf \set{ \big(\norm{w}{}^2 + \eps \norm{\nabla w}{}^2\big)^{1/2}}{w \in H^1(\Omega), w|_\Gamma = \widehat u}.
\end{align*}
Obviously, the latter norm is weaker than the previously defined corresponding boundary norm,
$\norm{\cdot}{H^{1/2}(\Gamma)}\le \norm{\cdot}{1/2,\Gamma}$.
The ultra-weak formulation from~\cite{HeuerK_RDM} is based on the spaces $U$ and $V$ defined by
\begin{align*}
  \widetilde U &:= L^2(\Omega) \times [L^2(\Omega)]^d \times L^2(\Omega) \times H^{1/2}(\cS) \times H^{1/2}(\cS) 
  \times H^{-1/2}(\cS) \times H^{-1/2}(\cS), \\
  U &:= \set{(u,\ssigma,\rho,\widehat u^a, \widehat u^b, \widehat\sigma^a,\widehat\sigma^b) \in \widetilde U}{
  \widehat u^a|_\Gamma = \widehat u^b|_\Gamma}, \\
  V &:= H^1(\TT) \times \HH(\div,\TT) \times H^1(\Delta,\TT), \quad\text{where}\\
  H^1(\Delta,\TT) &:= \{w\in H^1(\TT) \,:\, \Delta w|_T \in L^2(T) \,\forall T\in\TT\}.
\end{align*}
For the analysis we need two different norms in $U$,
\begin{align*}
  \norm{\uu}{U,1}^2 &:= \norm{u}{}^2 + \norm{\ssigma}{}^2 + \eps \norm{\rho}{}^2 \\
  &\qquad + \eps^{3/2}\norm{\widehat u^a}{1/2,\cS}^2 + \eps \norm{\widehat u^b}{1/2,\cS}^2
  + \eps^{3/2}\norm{\widehat \sigma^a}{-1/2,\cS}^2 + \eps^{5/2}\norm{\widehat \sigma^b}{-1/2,\cS}^2, \\
  \norm{\uu}{U,2}^2 &:= \norm{u}{}^2 + \norm{\ssigma}{}^2 + \eps \norm{\rho}{}^2 \\
  &\qquad + \norm{\widehat u^a}{1/2,\cS}^2 + \eps^{-1/2} \norm{\widehat u^b}{1/2,\cS}^2
  + \norm{\widehat \sigma^a}{-1/2,\cS}^2 + \eps^{1/2}\norm{\widehat \sigma^b}{-1/2,\cS}^2
\end{align*}
for $\uu=(u,\ssigma,\rho,\widehat u^a,\widehat u^b,\widehat\sigma^a,\widehat\sigma^b)\in U$.
These norms differ in their $\eps$-scalings of the skeleton components so that
$\norm{\cdot}{U,1}\le \norm{\cdot}{U,2}$. In both cases, field components are measured in
the so-called \emph{balanced norm} $(\norm{u}{}^2 + \norm{\ssigma}{}^2 + \eps \norm{\rho}{}^2)^{1/2}$
which, for the exact solution, is $(\norm{u}{}^2 + \eps^{1/2}\norm{\nabla u}{}^2 + \eps^{3/2} \norm{\Delta u}{}^2)^{1/2}$,
cf.~\cite{LinS_12_BFE,HeuerK_RDM}.

The test space $V$ is equipped with the norm
\begin{align*}
  \norm{\vv}V^2 &:= \eps^{-1}\norm{\mu}{}^2 + \norm{\pwnabla\mu}{}^2 + \eps^{-1/2} \norm{\ttau}{}^2 
  + \norm{\pwdiv\ttau}{}^2 \\
  &\qquad\qquad + \norm{v}{}^2 + (\eps^{1/2}+\eps)\norm{\pwnabla v}{}^2 + \eps^{3/2}\norm{\pwlap v}{}^2 \quad\text{for }
  \vv = (\mu,\ttau,v)\in V.
\end{align*}
This norm is induced by the inner product $\ip\cdot\cdot_V$ on $V$.
Note that this norm is equivalent to the one defined in~\cite{HeuerK_RDM}. The only difference is that the term
$\eps\norm{\pwnabla v}{}^2$ is not present in~\cite{HeuerK_RDM}. We use this norm here, to get a smaller constant in
Lemma~\ref{lem:sp:Hm12estimate} below.

\subsection{Ultra-weak formulation} \label{sec:sp:lap}

The ultra-weak formulation taken from~\cite{HeuerK_RDM} is derived by rewriting problem~\eqref{eq:intro:sigpde}
(with $c=\eps$) as a first order system,
\begin{align*}
  \rho-\div\ssigma = 0, \qquad \eps^{-1/4}\ssigma - \nabla u = 0, \qquad -\eps^{3/4}\rho + u = f.
\end{align*}
We test the first two equations with $\mu\in H^1(\TT)$ resp. $\ttau\in \HH(\div,\TT)$ element-wise, integrate by parts
and sum up over all elements.
The third equation is tested with $v-\eps^{1/2}\pwlap v$ for $v\in H^1(\Delta,\TT)$. Integrating by parts and using the
first two equations, we get the formulation
\begin{subequations}\label{eq:sp:weakform}
\begin{align}
  \ip{\rho}{\mu} + \ip{\ssigma}{\pwnabla\mu} - \dual{\widehat \sigma^a}{\mu}_\cS &= 0, \\
  \eps^{-1/4}\ip{\ssigma}{\ttau} + \ip{u}{\pwdiv\ttau} - \dual{\widehat u^a}{\ttau\cdot\nn}_\cS &= 0, \\
  \begin{split}
  \eps^{3/4}\ip{\ssigma}{\pwnabla v} - \eps^{3/4} \dual{\widehat\sigma^b}{v}_\cS + \ip{u}{v} \qquad\qquad\qquad& \quad
  \\
  +\eps^{5/4} \ip{\rho}{\pwlap v} + \eps^{1/4} \ip{\ssigma}{\pwnabla v} - \eps^{1/2} \dual{\widehat u^b}{\pwnabla
  v\cdot\nn}_\cS &= \ip{f}{v-\eps^{1/2}\pwlap v}.
\end{split}
\end{align}
\end{subequations}
The left and right-hand sides give rise to the definition of the bilinear form $b : U\times V \to \R$ and the linear
functional $L : V \to \R$
\begin{align*}
  b(\uu,\vv) &:= \ip{\rho}{\mu} + \ip{\ssigma}{\pwnabla\mu} - \dual{\widehat \sigma^a}{\mu}_\cS  + 
  \eps^{-1/4}\ip{\ssigma}{\ttau} + \ip{u}{\pwdiv\ttau} - \dual{\widehat u^a}{\ttau\cdot\nn}_\cS \\
  &\qquad +\eps^{3/4}\ip{\ssigma}{\pwnabla v} - \eps^{3/4} \dual{\widehat\sigma^b}{v}_\cS + \ip{u}{v} \\
  &\qquad +\eps^{5/4} \ip{\rho}{\pwlap v} + \eps^{1/4} \ip{\ssigma}{\pwnabla v} 
  - \eps^{1/2} \dual{\widehat u^b}{\pwnabla v\cdot\nn}_\cS,
  \\
  L(\vv) &:= \ip{f}{v-\eps^{1/2}\pwlap v},
\end{align*}
for all $\uu = (u,\ssigma,\rho,\widehat u^a,\widehat u^b, \widehat\sigma^a,\widehat\sigma^b) \in U$, $\vv =
(\mu,\ttau,v) \in V$.

The trial-to-test operator $\Theta_\beta : U \to V$ is defined as before, see~\eqref{eq:dpg:deftttop}.
Again, the operator $B:\;U \to V'$ is induced by the bilinear form $b(\cdot,\cdot)$
and~\eqref{eq:sp:weakform} can be written as
\begin{align*}
  \uu\in U:\quad B\uu = L.
\end{align*}
The non-trivial kernel of $B$ is related to the trace operators. For the present space $U$ we define them by
\begin{alignat*}{2}
  \gamma_0 &: U \to H^{1/2}(\Gamma), &\qquad \gamma_0\uu &:= \widehat u^a|_\Gamma, \\
  \gamma_\nn &: U \to H^{-1/2}(\Gamma), &\qquad \gamma_\nn\uu &:= \widehat\sigma^a|_\Gamma.
\end{alignat*}
For simplicity we only consider the symmetric formulation. The other cases can
be derived similarly, see also Section~\ref{sec:dpg}.
Analogously as in the unperturbed case we introduce the non-empty convex subset
\begin{align*}
  K^s = \set{ \uu\in U}{\gamma_0\uu\geq 0, \gamma_\nn\uu \geq 0},
\end{align*}
and define the bilinear form $a^s:\;U\times U\to\R$ and linear functional $F:\;U\to\R$ by
\begin{align*}
  a^s(\uu,\vv) &:= b(\uu,\Theta_\beta\vv) + \tfrac12 \eps^{1/4} ( \dual{\gamma_\nn\uu}{\gamma_0\vv}_\Gamma
  +\dual{\gamma_\nn\vv}{\gamma_0\uu}_\Gamma ), \\
  F(\vv) &:= L(\Theta_\beta\vv)
\end{align*}
for $\uu,\vv\in U$. Here, $\beta>0$ is a constant to be fixed.
Then, our variational inequality reads: Find $\uu\in K^s$ such that
\begin{align}\label{eq:sp:varineq}
  a^s(\uu,\vv-\uu) \geq F(\vv-\uu) \quad\text{for all } \vv\in K^s.
\end{align}
In the singularly perturbed case it is convenient to state coercivity and boundedness
of the bilinear form $a^s(\cdot,\cdot)$ in the \emph{energy-based} norm
\begin{align*}
  \enorm{\uu}^2 := \norm{B\uu}{V'}^2 + \eps^{-1/2} \norm{\gamma_0\uu}{H^{1/2}(\Gamma)}^2
  \quad (\uu\in U).
\end{align*}
Note that $\norm{B\,\cdot}{V'}$ is the \emph{energy norm} in standard DPG settings whereas
in our case it is a semi-norm.

Corresponding to Theorem~\ref{thm:dpg:main} we have the following result.

\begin{theorem}\label{thm:sp:main}
  For all $\beta\geq 3$ 
  the bilinear form $a^s:U\times U \to \R$ is coercive,
  \begin{align*} 
    \enorm{\uu}^2 \leq C_1 a^s(\uu,\uu) \quad\text{for all } \uu\in U,
  \end{align*}
  and bounded,
  \begin{align*} 
    |a^s(\uu,\vv)| \leq C_2 \enorm{\uu}\enorm{\vv} \quad\text{for all }\uu,\vv \in U.
  \end{align*}
  The constants $C_1,C_2>0$ do not depend on $\Omega$, $\TT$, or $\eps$.
  $C_1$ is independent of $\beta$ but $C_2$ is not.
  Furthermore,
  \begin{align}\label{eq:sp:normequiv}
    \norm{\uu}{U,1} \lesssim \enorm{\uu} \lesssim \norm{\uu}{U,2} \quad\text{for all } \uu\in U
  \end{align}
  with generic constants that are independent of $\TT$ and $\eps$.

  The variational inequality~\eqref{eq:sp:varineq} is uniquely solvable and equivalent to
  problem~\eqref{eq:intro:sigpde} (setting $c=\eps$) in the following sense: 
  If $u\in H^1(\Omega)$ solves problem~\eqref{eq:intro:sigpde},
  then $\uu = (u,\ssigma,\rho,\widehat u^a,\widehat u^b,\widehat\sigma^a,\widehat\sigma^b)\in K^s$ 
  with $\ssigma := \eps^{1/4}\nabla u$, $\widehat u^\star|_{\partial T} := u|_{\partial T}$,
  $\widehat\sigma^\star|_{\partial T} := \sigma\cdot\nn_T|_{\partial T}$
  for all $T\in\TT$ ($\star\in\{a,b\}$) solves~\eqref{eq:sp:varineq}. On the other hand,
  if $\uu = (u,\ssigma,\rho,\widehat u^a,\widehat u^b,\widehat\sigma^a,\widehat\sigma^b)\in K^s$ solves~\eqref{eq:sp:varineq}, 
  then $u\in H^1(\Omega)$ solves~\eqref{eq:intro:sigpde}.

  Moreover, the unique solution $\uu\in K^s$ of~\eqref{eq:sp:varineq} satisfies
  \begin{align*} 
    b(\uu,\Theta_\beta \ww) = F(\ww) \quad\text{for all } \ww\in U.
  \end{align*}
\end{theorem}

We prove this result in Section~\ref{sec:sp:proof}.

\subsection{Discretization, convergence and a posteriori error estimate} \label{sec:sp:disc}

We replace $U$ by the lowest order subspace
\begin{align*}
  U_h := P^0(\TT) \times [P^0(\TT)]^d \times P^0(\TT) \times S^1(\cS) \times S^1(\cS) \times 
  P^0(\cS) \times P^0(\cS)
\end{align*}
and $K^s$ by
\begin{align*}
  K_h^s := \set{\vv_h\in U_h}{\gamma_0\vv_h\geq 0,\gamma_\nn\vv_h\geq 0}.
\end{align*}
The discrete version of~\eqref{eq:sp:varineq} then reads: Find $\uu_h\in K_h^s$ such that
\begin{align}\label{eq:sp:varineqdisc}
  a^s(\uu_h,\vv_h-\uu_h) \geq F(\vv_h-\uu_h) \quad\text{for all } \vv_h \in K_h^s.
\end{align}
With the same arguments as in Section~\ref{sec:dpg:disc} we can prove unique solvability.

\begin{theorem} \label{thm:sp:disc}
  Under the same assumptions as in Theorem~\ref{thm:sp:main} the discrete variational
  inequality~\eqref{eq:sp:varineqdisc} admits a unique solution $\uu_h\in K_h^s$.
\end{theorem}

We have the following \emph{robust} quasi-optimal a priori error estimate.
Here, robustness means that the hidden constant does not depend on the positive perturbation
parameter $\eps$ (though, for simplicity we have assumed that $\eps\le 1$).

\begin{theorem} \label{thm:sp:apriori}
  For $\beta\ge 3$ let $\uu\in K^s$, $\uu_h\in K_h^s$ 
  denote the exact solutions of~\eqref{eq:sp:varineq}, \eqref{eq:sp:varineqdisc}.
  Then there holds 
  \begin{align*} 
    \norm{\uu-\uu_h}{U,1}^2 \lesssim 
    \inf_{\vv_h\in K_h^s} \left( \norm{\uu-\vv_h}{U,2}^2 
      + \tfrac12\eps^{1/4}(\dual{\gamma_\nn\uu}{\gamma_0(\vv_h-\uu)}_\Gamma 
      + \dual{\gamma_\nn(\vv_h-\uu)}{\gamma_0\uu}_\Gamma)\right).
  \end{align*}
  The generic constant does depend on $\Omega$ and $\beta$ but not on $\TT$ or $\eps$.
\end{theorem}

\begin{proof}
Following the proof of Theorem~\ref{thm:dpg:apriori} we obtain (by replacing $\norm{\cdot}U$ with $\enorm\cdot$)
\begin{align*}
  \enorm{\uu-\uu_h}^2 \lesssim 
    \inf_{\vv_h\in K_h^s} \left( \enorm{\uu-\vv_h}^2 
      + \tfrac12\eps^{1/4}(\dual{\gamma_\nn\uu}{\gamma_0(\vv_h-\uu)}_\Gamma 
      + \dual{\gamma_\nn(\vv_h-\uu)}{\gamma_0\uu}_\Gamma)\right).
\end{align*}
Then,~\eqref{eq:sp:normequiv} proves the error bound.
\end{proof}

The derivation of an a posteriori error estimate is analogous to Section~\ref{sec:aposteriori}.
For a function $\uu_h\in K_h^s$ we define local error indicators by
\begin{align*}
  \eta(T)^2 &:= \beta \norm{R_T^{-1} \iota_T^*(L-B\uu_h)}{V(T)}^2, \\
  \eta(E)^2 &:= \eps^{1/4} \dual{\gamma_\nn\uu_h}{\gamma_0\uu_h}_E,
\end{align*}
and the overall estimator
\begin{align*}
  \eta^2 := \sum_{T\in\TT} \eta(T)^2 + \sum_{E\in\cS_\Gamma} \eta(E)^2.
\end{align*}
Here, $V(T) := H^1(T) \times \HH(\div,T) \times H^1(\Delta,T)$ is equipped with the norm
\begin{align*}
  \norm{(\mu,\ttau,v)}{V(T)}^2 &:= \eps^{-1}\norm{\mu}{T}^2 + \norm{\nabla\mu}{T}^2 + \eps^{-1/2} \norm{\ttau}{T}^2 
  + \norm{\div\ttau}{T}^2 \\
  &\qquad\qquad + \norm{v}{T}^2 + (\eps^{1/2}+\eps)\norm{\nabla v}{T}^2 + \eps^{3/2}\norm{\Delta v}{T}^2,
\end{align*}
$R_T:\;V(T) \to (V(T))'$ denotes the Riesz isomorphism,
and $\iota_T^*$ is the dual of the canonical embedding $\iota_T:\;V(T)\to V$.

Analogously to the proof of Theorem~\ref{thm:dpg:aposteriori}, in conjunction with~\eqref{eq:sp:normequiv},
we obtain the following a posteriori estimate.
Like the a priori estimate from Theorem~\ref{thm:sp:apriori}, the a posteriori estimate
is \emph{robust} with respect to $\eps$ ($0<\eps\le 1$).

\begin{theorem} \label{thm:sp:aposteriori}
  For $\beta\ge 3$ let $\uu\in K^s$ and $\uu_h\in K_h^s$ be the solutions
  of~\eqref{eq:sp:varineq} and~\eqref{eq:sp:varineqdisc}, respectively. Then, there holds the reliability estimate
  \begin{align*} 
    \enorm{\uu-\uu_h}  \leq C_\mathrm{rel} \eta.
  \end{align*}
  The constant $C_\mathrm{rel}>0$ is independent of $\Omega$, $\TT$, $\beta$ and $\eps$. 
  In particular, with~\eqref{eq:sp:normequiv} we have
  \begin{align*}
    \norm{\uu-\uu_h}{U,1} \leq C_{\mathrm{rel},U} \eta,
  \end{align*}
  where $C_{\mathrm{rel},U} := C_\mathrm{rel} C_1$ and $C_1$ depends only on $\Omega$.
\end{theorem}

\subsection{Technical details} \label{sec:sp:tec}

Analogously as in Lemma~\ref{lem:trace}, we obtain boundedness of the trace operators.
In this case we use that, by definition of the norms,
$\norm{\widehat u^a|_\Gamma}{1/2,\Gamma}\leq \norm{\widehat u^a}{1/2,\cS}$ and
$\norm{\widehat\sigma^a|_\Gamma}{-1/2,\Gamma}\leq \norm{\widehat\sigma^a}{-1/2,\cS}$.

\begin{lemma} \label{lem:sp:trace}
  The operators $\gamma_0 : \big(U,\norm{\cdot}{U,2}\big) \to \big(H^{1/2}(\Gamma),\norm\cdot{1/2,\Gamma}\big)$ and 
  $\gamma_\nn : \big(U,\norm{\cdot}{U,2}\big) \to \big( H^{-1/2}(\Gamma),\norm\cdot{-1/2,\Gamma}\big)$ are bounded
  with constant $1$.
\end{lemma}

We now adapt the definition of the previously employed extension operator $\ext$ to the current situation.

For a function $\widehat v\in H^{1/2}(\Gamma)$ we define its quasi-harmonic extension
$\widetilde u \in H^1(\Omega)$ as the unique solution of
\begin{align}\label{eq:sp:harmext}
  -\eps \Delta \widetilde u + \widetilde u &=0 \quad\text{in }\Omega, \qquad
  \widetilde u|_\Gamma = \widehat v,
\end{align}
and define $\ext : H^{1/2}(\Gamma) \to U$ by
\begin{align*}
  \ext \widehat v := (\widetilde u, \eps^{1/4}\nabla \widetilde u, \eps^{1/4} \Delta \widetilde u, 
  \widetilde u|_{\cS}, \widetilde u|_{\cS}, \eps^{1/4} (\nabla\widetilde u\cdot\nn_T|_{\partial T})_{T\in\TT},
  \eps^{1/4} (\nabla\widetilde u\cdot\nn_T|_{\partial T})_{T\in\TT}).
\end{align*}
This operator characterizes the kernel of $B$.
Note that, in contrast to $\norm\cdot{1/2,\Gamma}$, the norm $\norm\cdot{H^{1/2}(\Gamma)}$
(recall the definitions in Section~\ref{sec:sp:notation})
is inherited from the energy norm associated to problem~\eqref{eq:sp:harmext}.

\begin{lemma}\label{lem:H12normequiv}
  For given $\widehat v\in H^{1/2}(\Gamma)$ let $\widetilde u \in H^1(\Omega)$ be the unique solution
  of~\eqref{eq:sp:harmext}. Then,
  \begin{align*}
    \norm{\widehat v}{H^{1/2}(\Gamma)}^2 = \eps\norm{\nabla \widetilde u}{}^2 + \norm{\widetilde u}{}^2 
    = \eps \dual{\partial_\nn \widetilde u}{\widehat v}_\Gamma = \eps^{3/4}\dual{\gamma_\nn \ext \widehat v}{\widehat v}_\Gamma.
  \end{align*}
\end{lemma}
\begin{proof}
  The last identity follows by definition of the operator $\ext$. Using the weak formulation of
  problem~\eqref{eq:sp:harmext} we have
  \begin{align*}
    \eps\norm{\nabla \widetilde u}{}^2 + \norm{\widetilde u}{}^2 &= \eps\dual{\partial_\nn\widetilde u}{\widehat v}_\Gamma 
    = \eps\ip{\nabla \widetilde u}{\nabla w} + \ip{\widetilde u}{w} \\
    &\leq \Big( \eps\norm{\nabla \widetilde u}{}^2 + \norm{\widetilde u}{}^2\Big)^{1/2} 
    \Big( \eps\norm{\nabla w}{}^2 + \norm{w}{}^2\Big)^{1/2}
  \end{align*}
  for all $w\in H^1(\Omega)$ with $w|_\Gamma = \widehat v$. 
  Thus, $\norm{\widehat v}{H^{1/2}(\Gamma)}^2 = \eps\norm{\nabla \widetilde u}{}^2 + \norm{\widetilde u}{}^2$.
\end{proof}

Similarly to Lemma~\ref{lem:kerBext} there holds the following result in the singularly perturbed case.

\begin{lemma}\label{lem:sp:kerBext}
  The operators $B : U \to V'$, $\ext: H^{1/2}(\Gamma)\to U$ have the following properties:
  \begin{enumerate}[(i)]
    \item $B$ and $\ext$ are bounded,
      \begin{align*}
        \norm{B\uu}{V'} &\lesssim \norm{\uu}{U,2} \quad\text{for all }\uu\in U, \qquad
        \norm{\ext \widehat v}{U,1} \simeq \eps^{-1/4} \norm{\widehat v}{H^{1/2}(\Gamma)} 
        \quad\text{for all } \widehat v \in H^{1/2}(\Gamma).
      \end{align*}
      The generic constants are independent of $\TT$ and $\eps$.
    \item\label{lem:sp:kerBext:b} The kernel of $B$ consists of all quasi-harmonic extensions,
                                  i.e., $\ker(B) = \ran(\ext)$.
    \item\label{lem:sp:kerBext:c} $\ext$ is a right-inverse of $\gamma_0$.
    \item\label{lem:sp:kerBext:d} $B:\;U/\ker(B)\to V'$ is inf-sup stable,
      \begin{align*}
        \norm{\uu-\ext\gamma_0\uu}{U,1} \lesssim \norm{B\uu}{V'} = \sup_{\vv\in V} \frac{\dual{B\uu}{\vv}}{\norm{\vv}{V}}
        = b(\uu,\Theta\uu)^{1/2} \quad\text{for all } \uu\in U.
      \end{align*}
      The generic constant depends on $\Omega$ but not on $\TT$ or $\eps$.
  \end{enumerate}
\end{lemma}

\begin{proof}
First, by~\cite[Lemma~3]{HeuerK_RDM} we have $b(\uu,\vv)\lesssim\norm{\uu}{U,2}\norm{\vv}V$ for all $\uu\in U,\vv\in V$.
Dividing by $\norm{\vv}V$ and taking the supremum proves boundedness of $B$.

Second, for $\widehat v\in H^{1/2}(\Gamma)$, let
$\uu = (u,\ssigma,\rho,\widehat u^a,\widehat u^b,\widehat\sigma^a,\widehat\sigma^b) = \ext \widehat v$.
By definition of the skeleton norms we have
\begin{align*}
  &\norm{\widehat u^\star}{1/2,\cS}^2 \leq \norm{u}{}^2 + \norm{\ssigma}{}^2,\quad
  \norm{\widehat \sigma^\star}{-1/2,\cS}^2 \leq \norm{\ssigma}{}^2 + \eps\norm{\rho}{}^2
  \quad (\star\in\{a,b\}).
\end{align*}
In the definition of $\norm\cdot{U,1}$, these norms are scaled with positive powers of $\eps$.
Thus,
\begin{align*}
  \norm{u}{}^2 + \norm{\ssigma}{}^2 + \eps \norm{\rho}{}^2 \leq \norm{\uu}{U,1}^2 \lesssim 
  \norm{u}{}^2 + \norm{\ssigma}{}^2 + \eps \norm{\rho}{}^2.
\end{align*}
Using that $\rho = \eps^{1/4} \Delta u = \eps^{1/4} \eps^{-1} u = \eps^{-3/4} u$ and $\ssigma = \eps^{1/4}\nabla u$
by the definition of $\uu=\ext\widehat v$, cf.~\eqref{eq:sp:harmext}, we obtain
\begin{align*}
  \norm{u}{}^2 + \norm{\ssigma}{}^2 + \eps\norm{\rho}{}^2 = \norm{u}{}^2 + \eps^{1/2} \norm{\nabla u}{}^2 + \eps^{-1/2}
  \norm{u}{}^2 \simeq \eps^{-1/2}\left( \norm{u}{}^2 + \eps \norm{\nabla u}{}^2 \right)
\end{align*}
The last term on the right-hand side is equal to $\eps^{-1/2}\norm{\widehat v}{H^{1/2}(\Gamma)}^2$
by Lemma~\ref{lem:H12normequiv}.

Finally,~\eqref{lem:sp:kerBext:b},~\eqref{lem:sp:kerBext:c} and~\eqref{lem:sp:kerBext:d} 
are proven in~\cite[Lemmas~3,4]{FuehrerH_RCD}.
\end{proof}

Similarly to Lemma~\ref{lem:dpg:Hm12estimate} (for the unperturbed case) we need to control
the Neumann traces of elements of the quotient space $U/\ker(B)$.
As we have seen, this has a fundamental relation to the stability of the homogeneous adjoint
problem with prescribed Dirichlet boundary condition, cf.~Lemma~\ref{lem:dpg:dualestimate}.
In the singularly perturbed case the situation
is a little more technical. Following \cite{HeuerK_RDM} (see Lemmas~8 and~9 there),
we split the stability analysis of the adjoint problem into two parts. These are
the following Lemmas~\ref{lem:sp:dualestimate} and~\ref{lem:sp:dualproblem}. The
last lemma of this section (Lemma~\ref{lem:sp:Hm12estimate}) then states the control
of the Neumann traces.

\begin{lemma}\label{lem:sp:dualestimate}
  Let $\widehat w\in H^{1/2}(\Gamma)$ be given. The problem
  \begin{subequations}\label{eq:sp:dualestimate}
  \begin{align}
    \div\llambda + \eps^{-1/2}w &= 0 \quad\text{in } \Omega, \label{eq:sp:dualestimate:a} \\
    \llambda + \nabla w &= 0 \quad\text{in } \Omega, \label{eq:sp:dualestimate:b} \\
    w|_{\Gamma} &= \widehat w \label{eq:sp:dualestimate:c}
  \end{align}
  \end{subequations}
  admits a unique solution $(w,\llambda) \in H^1(\Omega)\times \HH(\div,\Omega)$ with $\Delta w\in H^1(\Omega)$, and
  \begin{align*}
    \norm{\llambda}{}^2 + \eps^{1/2}\norm{\div\llambda}{}^2 
    = \norm{\nabla w}{}^2 + \eps^{-1/2}\norm{w}{}^2 \leq 
    \eps^{-1} \norm{\widehat w}{H^{1/2}(\Gamma)}^2.
  \end{align*}
\end{lemma}

\begin{proof}
The proof follows the same arguments as given in the proof of Lemma~\ref{lem:dpg:dualestimate}.
To see the estimate for the norms, we make use of the weak formulation and obtain
\begin{align*}
  \norm{\nabla w}{}^2 + \eps^{-1/2}\norm{w}{}^2 \leq 
  \norm{\nabla \widetilde u}{}^2  + \eps^{-1/2}\norm{\widetilde u}{}^2 \leq 
  \eps^{-1} \left( \eps\norm{\nabla\widetilde u}{}^2 + \norm{\widetilde u}{}^2 \right)
\end{align*}
for all $\widetilde u\in H^1(\Omega)$ with $\widetilde u|_\Gamma = \widehat w$.
Taking the infimum over these functions $\widetilde u$ finishes the proof.
\end{proof}

\begin{lemma}\label{lem:sp:dualproblem}
  Let $\widehat v\in H^{1/2}(\Gamma)$ be given. The problem
  \begin{subequations}\label{eq:sp:dualproblem}
  \begin{align}
    \div\ttau + v &= 0\quad\text{in } \Omega, \label{eq:sp:dualproblem:a} \\
    \nabla\mu + (\eps^{1/4}+\eps^{3/4})\nabla v + \eps^{-1/4}\ttau &= 0
                      \quad\text{in } \Omega, \label{eq:sp:dualproblem:b}\\
    \eps^{5/4} \Delta v + \mu &= 0\quad\text{in } \Omega, \label{eq:sp:dualproblem:c} \\
    \quad v|_\Gamma &= 0, \qquad \Delta v|_\Gamma = -\eps^{-5/4} \widehat v \label{eq:sp:dualproblem:d}
  \end{align}
  \end{subequations}
  has a unique solution $\vv:=(\mu,\ttau,v)\in H^1(\Omega)\times \HH(\div,\Omega)\times H^1(\Delta,\Omega)$.
  It satisfies $\mu|_\Gamma = \widehat v$ and
  \begin{align*}
    \norm{\vv}{V} \leq 3/\sqrt{2} \eps^{-1/2}\norm{\widehat v}{H^{1/2}(\Gamma)}.
  \end{align*}
\end{lemma}

\begin{proof}
Let $(w,\llambda)\in H^1(\Omega)\times \HH(\div,\Omega)$ be the solution
of~\eqref{eq:sp:dualestimate} with $\widehat w = \eps^{-1/4}\widehat v$.
Define $v\in H_0^1(\Omega)$ to be the unique solution of
\begin{align*}
  -\eps \Delta v + v = w\quad\text{in } \Omega, \qquad v|_\Gamma = 0.
\end{align*}
Note that $\Delta v = \eps^{-1} (v-w) \in H^1(\Omega)$
and $\Delta v|_\Gamma = -\eps^{-1}w|_\Gamma = -\eps^{-5/4}\widehat v$.
We define $\mu:= -\eps^{5/4}\Delta v\in H^1(\Omega)$ and have
that $\mu|_\Gamma = \eps^{1/4}w|_\Gamma = \widehat v$.
In particular, \eqref{eq:sp:dualproblem:c} and \eqref{eq:sp:dualproblem:d} are satisfied.

Now, define $\ttau := \eps^{1/2}\llambda - \eps\nabla v$.
Note that $\llambda\in \HH(\div,\Omega)$ and $\nabla v \in \HH(\div,\Omega)$. Hence, $\ttau\in \HH(\div,\Omega)$.
Together with~\eqref{eq:sp:dualestimate:a} and $-\eps\Delta v + v = w$ we get
\begin{align*}
  \div\ttau = \eps^{1/2}\div\llambda - \eps\Delta v = -w - \eps\Delta v = -v,
\end{align*}
which is~\eqref{eq:sp:dualproblem:a}.
One also establishes that \eqref{eq:sp:dualproblem:b} holds. In fact, by the definition of $\ttau$
and relation \eqref{eq:sp:dualestimate:b}, we find
\begin{align*}
   \nabla\mu + (\eps^{1/4}+\eps^{3/4})\nabla v + \eps^{-1/4}\ttau
   &=
   \nabla\mu + (\eps^{1/4}+\eps^{3/4})\nabla v + \eps^{-1/4} (\eps^{1/2}\llambda-\eps\nabla v)\\
   &=
   \nabla\mu + \eps^{1/4}\nabla v - \eps^{1/4}\nabla w.
\end{align*}
The last term vanishes since $\mu=-\eps^{5/4}\Delta v = -\eps^{1/4}(v-w)$ by definition of $\mu$ and $v$.

Now, testing~\eqref{eq:sp:dualproblem:c} with $\eps^{1/4}\Delta z$,~\eqref{eq:sp:dualproblem:b}
with $\eps^{1/4}\nabla z$, and~\eqref{eq:sp:dualproblem:a} with $z$ for
$z\in H^1(\Delta,\Omega)$ with $z|_\Gamma = 0$, and adding the resulting equations,
we obtain, after integrating by parts,
\begin{align}\label{eq:sp:dualproblem:varform}
  \eps^{3/2}\ip{\Delta v}{\Delta z} + (\eps^{1/2}+\eps) \ip{\nabla v}{\nabla z} + \ip{v}{z} = 
  -\eps^{1/4} \dual{\nabla z\cdot\nn_\Omega}{\widehat v}_\Gamma.
\end{align}
For any $\widetilde u\in H^1(\Omega)$ with $\widetilde u|_\Gamma = \widehat v$ we infer
\begin{align*}
  \eps^{3/2} \norm{\Delta v}{}^2 + (\eps^{1/2}+\eps)\norm{\nabla v}{}^2 + \norm{v}{}^2 &=
  -\eps^{1/4}\dual{\nabla v\cdot\nn_\Omega}{\widehat v}_\Gamma
  = -\eps^{1/4} \ip{\Delta v}{\widetilde u} -\eps^{1/4}\ip{\nabla v}{\nabla\widetilde u} \\
  &\leq \eps^{3/4}\norm{\Delta v}{} \norm{\eps^{-1/2}\widetilde u}{} + \eps^{1/4}\norm{\nabla v}{} 
  \norm{\nabla \widetilde u}{} \\
  &\leq \left(\eps^{3/2}\norm{\Delta v}{}^2 + \eps^{1/2}\norm{\nabla v}{}^2\right)^{1/2}
  \eps^{-1/2}\left( \norm{\widetilde u}{}^2 + \eps \norm{\nabla \widetilde u}{}^2\right)^{1/2}.
\end{align*}
On the one hand, we conclude
\begin{align}\label{eq:Hm12estimate:estV1}
  \left(\eps^{3/2} \norm{\Delta v}{}^2 + (\eps^{1/2}+\eps)\norm{\nabla v}{}^2 + \norm{v}{}^2\right)^{1/2} 
  \leq \eps^{-1/2}\norm{\widehat v}{H^{1/2}(\Gamma)}.
\end{align}
On the other hand, using Young's inequality, we also conclude that
\begin{align*}
  \eps^{3/2} \norm{\Delta v}{}^2 + (\eps^{1/2}+\eps)\norm{\nabla v}{}^2 + \norm{v}{}^2 
  \leq \tfrac{\delta^{-1}}2 \left(\eps^{3/2} \norm{\Delta v}{}^2 + \eps^{1/2}\norm{\nabla v}{}^2\right) 
  + \tfrac{\delta}2 \eps^{-1} \norm{\widehat v}{H^{1/2}(\Gamma)}^2.
\end{align*}
For $\delta=\tfrac12$ we get
\begin{align}\label{eq:Hm12estimate:estV2}
  \eps\norm{\nabla v}{}^2 + \norm{v}{}^2 \leq \frac14 \eps^{-1}\norm{\widehat v}{H^{1/2}(\Gamma)}^2.
\end{align}
By~\eqref{eq:sp:dualproblem:c} and~\eqref{eq:sp:dualproblem:a} we have
\begin{align*}
  \eps^{-1}\norm{\mu}{}^2 + \norm{\div\ttau}{}^2 = \eps^{3/2}\norm{\Delta v}{}^2 
  + \norm{v}{}^2.
\end{align*}
It remains to estimate the norms of $\ttau$ and $\nabla v$.
To this end we rewrite the term $\ip{\llambda}{\nabla v}$.
Integrating by parts, the condition $v|_\Gamma = 0$,~\eqref{eq:sp:dualestimate:a},
and the identity $w=-\eps\Delta v + v$ show that
\begin{align*}
  \ip{\llambda}{\nabla v} = -\ip{\div\llambda}{v}
  &= \eps^{-1/2}\ip{w}{v} = -\eps^{1/2}\ip{\Delta v}{v} + \eps^{-1/2}\ip{v}{v}
   = \eps^{-1/2}\left( \eps\norm{\nabla v}{}^2 + \norm{v}{}^2\right).
\end{align*}
Recall that $\eps^{-1/4}\ttau = \eps^{1/4}\llambda -
\eps^{3/4}\nabla v$. Thus,
\begin{align*}
  \eps^{-1/2}\norm{\ttau}{}^2 = \eps^{1/2}\norm{\llambda}{}^2 
  -2\eps\ip{\llambda}{\nabla v} + \eps^{3/2}\norm{\nabla v}{}^2.
\end{align*}
For the estimation of $\norm{\nabla \mu}{}$ we use~\eqref{eq:sp:dualproblem:b} 
and again $\eps^{1/4}\llambda = \eps^{3/4}\nabla v + \eps^{-1/4}\ttau$ to obtain
\begin{align*}
  \norm{\nabla\mu}{}^2 = \norm{\eps^{1/4}\nabla v + \eps^{1/4}\llambda}{}^2 
  = \eps^{1/2}\norm{\nabla v}{}^2 + 2\eps^{1/2}\ip{\llambda}{\nabla v} 
  + \eps^{1/2}\norm{\llambda}{}^2.
\end{align*}
This gives the estimate
\begin{align*}
  &\eps^{-1}\norm{\mu}{}^2 + \norm{\div\ttau}{}^2 +
  \eps^{-1/2}\norm{\ttau}{}^2 + \norm{\nabla\mu}{}^2 
  \\ &\qquad\qquad \leq 
  \eps^{3/2}\norm{\Delta v}{}^2 + \norm{v}{}^2 +
  2\eps^{1/2}\norm{\llambda}{}^2 +
  (\eps^{1/2}+2\eps)\norm{\nabla v}{}^2 + 2\norm{v}{}^2
  \\ &\qquad\qquad \leq 
  \eps^{3/2}\norm{\Delta v}{}^2 + (\eps^{1/2}+\eps)\norm{\nabla v}{}^2 + \norm{v}{}^2 + 
  2\left(\eps\norm{\nabla v}{}^2 + \norm{v}{}^2\right) + 2\eps^{1/2}\norm{\llambda}{}^2.
\end{align*}
Using estimates~\eqref{eq:Hm12estimate:estV1},~\eqref{eq:Hm12estimate:estV2}, and Lemma~\ref{lem:sp:dualestimate},
we put everything together to conclude for the overall norm
\begin{align*}
  \norm{(\mu,\ttau,v)}{V}^2 \leq \frac{9}2 \eps^{-1}\norm{\widehat v}{H^{1/2}(\Gamma)}^2.
\end{align*}
To see uniqueness of $\vv$, let $(\mu_2,\ttau_2,v_2)$ solve~\eqref{eq:sp:dualproblem}.
Define $\ww:=\vv-(\mu_2,\ttau_2,v_2)$ and set $w:=v-v_2$. Note that $w|_\Gamma = 0$ and 
$\Delta w|_\Gamma = 0$. The variational formulation for $w$, which can be obtained in the same way
as~\eqref{eq:sp:dualproblem:varform}, proves
\begin{align*}
  \eps^{3/2} \norm{\Delta w}{}^2 + (\eps^{1/2}+\eps)\norm{\nabla w}{}^2 + \norm{w}{}^2 = 0.
\end{align*}
Thus, $w=0$ or equivalently $v=v_2$. Equation~\eqref{eq:sp:dualproblem:c} then gives $\mu = -\eps^{5/4}\Delta v = 
-\eps^{5/4}\Delta v_2 = \mu_2$ and, similarly,~\eqref{eq:sp:dualproblem:b} shows $\ttau=\ttau_2$.
\end{proof}

\begin{lemma}\label{lem:sp:Hm12estimate}
  There holds
  \begin{align*}
    \eps^{1/4}|\dual{\gamma_\nn(\uu-\ext\gamma_0\uu)}{\widehat v}_\Gamma| \leq 3/\sqrt{2}
    \norm{B\uu}{V'}\eps^{-1/4}\norm{\widehat v}{H^{1/2}(\Gamma)} \quad\text{for all } \uu\in U, \,
    \widehat v\in H^{1/2}(\Gamma).
  \end{align*}
\end{lemma}
\begin{proof}
The idea of the proof is the same as in the proof of Lemma~\ref{lem:dpg:Hm12estimate}, using
Lemma~\ref{lem:sp:dualproblem} instead of Lemma~\ref{lem:dpg:dualestimate}.
\end{proof}

\subsection{Proof of Theorem~\ref{thm:sp:main}} \label{sec:sp:proof}

First, we show boundedness and coercivity of $a^s(\cdot,\cdot)$ with respect to the norm $\enorm\cdot$.
Then the Lions-Stampacchia theorem, see, e.g.,~\cite{Glowinski_08_NMN,GlowinskiLT_81_NAV,Rodrigues_87_OPM},
proves unique solvability of~\eqref{eq:dpg:varineq} provided that $F : U\to \R$ is a linear functional, 
which follows by the boundedness of $\Theta:U\to\R$.

We start by showing the boundedness.
Since $b(\cdot,\Theta\cdot)$ is symmetric and positive semi-definite, the Cauchy-Schwarz inequality proves
\begin{align*}
  |b(\uu,\Theta_\beta\vv)| \leq \beta b(\uu,\Theta\uu)^{1/2} b(\vv,\Theta\vv)^{1/2} = \beta \norm{B\uu}{V'}\norm{B\vv}{V'}
\end{align*}
for all $\uu,\vv\in U$. For the boundary terms, we consider
\begin{align*}
  \eps^{1/4} |\dual{\gamma_\nn\uu}{\gamma_0\vv}_\Gamma| \leq
  \eps^{1/4}|\dual{\gamma_\nn(\uu-\ext\gamma_0\uu)}{\gamma_0\vv}_\Gamma| 
  + \eps^{1/4} |\dual{\gamma_\nn\ext\gamma_0\uu}{\gamma_0\vv}_\Gamma|.
\end{align*}
The first term on the right-hand side is estimated with Lemma~\ref{lem:sp:Hm12estimate}. 
Let $\widetilde u \in H^1(\Omega)$ be the quasi-harmonic extension of $\gamma_0\uu$ and let $\widetilde v \in
H^1(\Omega)$ be the quasi-harmonic extension of $\gamma_0\vv$. 
Then, for the second term we get with integration by parts (cf.~Lemma~\ref{lem:H12normequiv})
\begin{align*}
  \eps^{1/4}\dual{\gamma_\nn\ext\gamma_0\uu}{\gamma_0\vv}_\Gamma = \eps^{-1/2} 
  \Big( \eps\ip{\nabla \widetilde u}{\nabla\widetilde v} + \ip{\widetilde u}{\widetilde v} \Big)
  \leq \eps^{-1/4}\norm{\gamma_0\uu}{H^{1/2}(\Gamma)} \eps^{-1/4}\norm{\gamma_0\vv}{H^{1/2}(\Gamma)}.
\end{align*}
Together, this gives for the boundary term
\begin{align*}
  \eps^{1/4} |\dual{\gamma_\nn\uu}{\gamma_0\vv}_\Gamma| \leq 
  3/\sqrt{2} \norm{B\uu}{V'}\eps^{-1/4}\norm{\gamma_0\vv}{H^{1/2}(\Gamma)} + 
  \eps^{-1/4}\norm{\gamma_0\uu}{H^{1/2}(\Gamma)} \eps^{-1/4}\norm{\gamma_0\vv}{H^{1/2}(\Gamma)}.
\end{align*}
The second boundary term $\eps^{1/4} \dual{\gamma_\nn\vv}{\gamma_0\uu}_\Gamma$ is treated identically.
Altogether, this proves the boundedness of $a^s(\cdot,\cdot)$.

For the proof of coercivity we use Lemma~\ref{lem:H12normequiv}, Lemma~\ref{lem:sp:Hm12estimate} 
and Young's inequality to find that, for $\delta>0$,
\begin{align*}
  \enorm{\uu}^2 &= \norm{B\uu}{V'}^2 + \eps^{-1/2} \norm{\gamma_0\uu}{H^{1/2}(\Gamma)}^2 = 
  \norm{B\uu}{V'}^2 + \eps^{1/4}\dual{\gamma_\nn\ext\gamma_0\uu}{\gamma_0\uu}_\Gamma  \\
  &= \norm{B\uu}{V'}^2 + \eps^{1/4}\dual{\gamma_\nn(\ext\gamma_0\uu-\uu)}{\gamma_0\uu}_\Gamma + 
  \eps^{1/4}\dual{\gamma_\nn\uu}{\gamma_0\uu}_\Gamma  \\
  &\leq \norm{B\uu}{V'}^2 + 3/\sqrt{2} \norm{B\uu}{V'} \eps^{-1/4}\norm{\gamma_0\uu}{H^{1/2}(\Gamma)} + 
  \eps^{1/4}\dual{\gamma_\nn\uu}{\gamma_0\uu}_\Gamma  \\
  &\leq (1+\delta^{-1}\tfrac{9}4) \norm{B\uu}{V'}^2 + \eps^{1/4}\dual{\gamma_\nn\uu}{\gamma_0\uu}_\Gamma
  +\tfrac\delta2 \eps^{-1/2}\norm{\gamma_0\uu}{H^{1/2}(\Gamma)}^2.
\end{align*}
We choose $\delta = \tfrac{9}8$, which implies $1+\delta^{-1}\tfrac{9}4 = 3$.
Then, subtracting the last term on the right-hand side, we obtain for $\beta\ge 3$
\begin{align*}
  \norm{B\uu}{V'}^2 + \tfrac{7}{16} \eps^{-1/2} \norm{\gamma_0\uu}{H^{1/2}(\Gamma)}^2 
  \leq 3 \norm{B\uu}{V'}^2 + \eps^{1/4}\dual{\gamma_\nn\uu}{\gamma_0\uu}_\Gamma 
  \leq a^s(\uu,\uu) \quad\text{for all }\uu\in U.
\end{align*}
Next we show~\eqref{eq:sp:normequiv}. By Lemma~\ref{lem:sp:kerBext} and the triangle inequality we obtain
\begin{align*}
  \norm{\uu}{U,1}^2 \lesssim \norm{\uu-\ext\gamma_0\uu}{U,1}^2 + \norm{\ext\gamma_0\uu}{U,1}^2 \lesssim
  \norm{B\uu}{V'}^2 + \eps^{-1/2}\norm{\gamma_0\uu}{H^{1/2}(\Gamma)}^2.
\end{align*}
For the proof of the upper bound in~\eqref{eq:sp:normequiv} let $\widetilde u\in H^1(\Omega)$ be a function which
attains the minimum in the definition of $\norm{\widehat u^b}{1/2,\cS}$. 
Then,
\begin{align*}
  \eps^{-1/2} \norm{\widehat u^b}{1/2,\cS}^2 = \eps^{-1/2}\big( \norm{\widetilde u}{}^2 + 
  \eps^{1/2}\norm{\nabla \widetilde u}{}^2 \big) \geq \eps^{-1/2} \big( \norm{\widetilde u}{}^2 +
  \eps\norm{\nabla \widetilde u}{}^2 \big).
\end{align*}
Since $\widetilde u|_\Gamma = \widehat u^b|_\Gamma = \widehat u^a|_\Gamma = \gamma_0\uu$, the right-hand side is an
upper bound of $\eps^{-1/2}\norm{\gamma_0\uu}{H^{1/2}(\Gamma)}^2$.
Thus, together with Lemma~\ref{lem:sp:kerBext}, we get
\begin{align*}
  \enorm{\uu}^2 = \norm{B\uu}{V'}^2 + \eps^{-1/2}\norm{\gamma_0\uu}{H^{1/2}(\Gamma)}^2 \lesssim \norm{\uu}{U,2}^2 +
  \eps^{-1/2}\norm{\widehat u^b}{1/2,\cS}^2 \lesssim \norm{\uu}{U,2}^2.
\end{align*}
The remainder of the proof follows the same arguments as in the proof of Theorem~\ref{thm:dpg:main}.

\section{Examples}\label{sec:examples}
In this section we present various numerical examples in two dimensions ($d=2$).
For the first example in \S\ref{sec:examples:smooth} we take a manufactured solution $u\in H^2(\Omega)$.
The standard finite element method with lowest-order discretization on quasi-uniform meshes converges
at a rate $\OO(h)$, where $h$ denotes the diameter of elements in $\TT$.
We observe the same optimal rate for the DPG methods with $\star\in\{0,\nn,s\}$, analyzed in Section~\ref{sec:dpg}.
In \S\ref{sec:examples:Lshape} we consider an L-shaped domain with unknown solution and an expected singularity
at the reentrant corner. Indeed, we will see that a uniform method gives suboptimal convergence whereas an adaptive
method driven by the estimator $\eta$ from \S\ref{sec:aposteriori} recovers the optimal one.
Finally, in \S\ref{sec:examples:RDsmooth} we use a family of manufactured solutions that exhibit
typical boundary layers for the reaction dominated diffusion problem. Our numerical results underline
the robustness of the a posteriori error estimate, as stated by Theorem~\ref{thm:sp:aposteriori}.

\subsection{General setting}
As is usual for DPG methods we replace the infinite dimensional test space $V$ used in the calculation of optimal test
functions~\eqref{eq:dpg:deftttop} by a finite dimensional subspace $V_h$, that is,
we replace the test function $\Theta_\beta\uu_h$ for $\uu_h\in U_h$ by $\Theta_{\beta,h}\uu_h$ defined through 
\begin{align*}
  \ip{\Theta_{\beta,h}\uu_h}{\vv_h}_V = \beta b(\uu_h,\vv_h) \quad\text{for all }\vv_h\in V_h.
\end{align*}
Here we choose
\begin{align*}
  V_h := \begin{cases}
    P^2(\TT)\times[P^2(\TT)]^2 & \text{for the methods from Section~\ref{sec:dpg}}, \\
    P^2(\TT)\times[P^2(\TT)]^2 \times P^4(\TT) & \text{for the method from Section~\ref{sec:sp}}.
  \end{cases}
\end{align*}
These choices are motivated by~\cite{GopalakrishnanQ_14_APD}.
We refer the interested reader to this work for more details.
The resulting DPG scheme is called \emph{practical} DPG method.
For the scaling parameter of the test functions we choose $\beta=2$ for the methods from Section~\ref{sec:dpg} and
$\beta=3$ for the method from Section~\ref{sec:sp}.

We use the standard basis for the lowest order spaces $U_h$, that is, the element characteristic
functions for $P^0(\TT)$, $P^0(\cS)$, and nodal basis functions (hat-functions) for $S^1(\cS)$.
These choices allow for a simple implementation of the inequality constraints in the cones $K_h^\star$.

We solve the discrete variational inequalities~\eqref{eq:dpg:varineqdisc},~\eqref{eq:sp:varineqdisc} with a
\emph{(Primal-Dual) Active Set Algorithm}, see~\cite{HoppeK_94_AMM,KarkkainenKT_03_ALA}.
More precisely, we implemented a modification of~\cite[Algorithm A1]{KarkkainenKT_03_ALA} (which deals with obstacle
problems) to the present problem (here we consider inequality constraints only for degrees of freedom that are
associated to the boundary).

For the problems where the solution is known in analytical form we compute different error quantities depending
on the underlying problem from Section~\ref{sec:dpg} or~\ref{sec:sp}.
\begin{itemize}
  \item Section~\ref{sec:dpg}: Let $\uu=(u,\ssigma,\widehat u,\widehat\sigma)$ denote the exact solution
    of~\eqref{eq:dpg:varineq} and let $\uu_h=(u_h,\ssigma_h,\widehat u_h,\widehat\sigma_h)$ be its approximation.
    We define
    \begin{alignat*}{2}
      \quad\qquad\err(u) &:= \norm{u-u_h}{}, &\quad \err(\ssigma) &:= \norm{\ssigma-\ssigma_h}{}, \\
      \quad\qquad\err(\widehat u) &:= \left(\norm{u-\widetilde u_h}{}^2 + \norm{\nabla(u-\widetilde u_h)}{}^2\right)^{1/2},
      &\quad \err(\widehat\sigma) &:= \left(\norm{\ssigma-\widetilde\ssigma_h}{}^2 +
      \norm{\div(\ssigma-\widetilde\ssigma_h)}{}^2\right)^{1/2}.
    \end{alignat*}
    Here, $\widetilde u_h\in S^1(\TT)$ is the nodal interpolant of $\widehat u_h$ at the nodes of $\TT$.
    Similarly, $\widetilde\ssigma_h$ is the Raviart-Thomas interpolation of $\widehat \sigma_h$.
    Then, it follows by the definition of the trace norms
    \begin{align*}
      \norm{\uu-\uu_h}{U} \leq \left(\err(u)^2 + \err(\ssigma)^2 + \err(\widehat u)^2 +
      \err(\widehat\sigma)^2\right)^{1/2}.
    \end{align*}
  \item Section~\ref{sec:sp}: Let $\uu=(u,\ssigma,\rho,\widehat u^a,\widehat u^b,\widehat\sigma^a,\widehat\sigma^b)$ 
    denote the exact solution
    of~\eqref{eq:sp:varineq} and let 
    $\uu_h=(u_h,\ssigma_h,\rho_h,\widehat u_h^a,\widehat u_h^b,\widehat\sigma_h^a,\widehat\sigma_h^b)$ be its approximation. 
    Define $\err(u)$ and $\err(\ssigma)$ as above and additionally
    \begin{align*}
      \err(\widehat u^\star) 
      &:= \left(\norm{u-\widetilde u_h^\star}{}^2 + \eps^{1/2}\norm{\nabla(u-\widetilde u_h^\star)}{}^2\right)^{1/2}, \\
      \err(\widehat\sigma^\star) &:= 
      \left(\norm{\ssigma-\widetilde\ssigma_h^\star}{}^2 + \eps\norm{\div(\ssigma-\widetilde\ssigma_h^\star)}{}^2\right)^{1/2}, 
      \\
      \err(\rho) &:= \eps^{1/2}\norm{\rho-\rho_h}{},
    \end{align*}
    for $\star\in\{a,b\}$, and $\widetilde u_h^\star$, $\widetilde\sigma_h^\star$ are defined in the same way as above. 
    Our total error estimator is
    \begin{align*}
      \err(\uu) &:= \Big( \err(u)^2 + \err(\ssigma)^2 + \err(\rho)^2 \\
      &\qquad\qquad +\eps^{3/2} \err(\widehat u^a)^2 + \eps \err(\widehat u^b)^2 
      + \eps^{3/2}\err(\widehat\sigma^a)^2+\eps^{5/2}\err(\widehat\sigma^b)^2 \Big)^{1/2}
    \end{align*}
    so that
    \begin{align*}
      \norm{\uu-\uu_h}{U,1} \leq \err(\uu).
    \end{align*}
\end{itemize}
For examples with singularities and/or boundary or interior layers we use a standard adaptive
algorithm that uses $\eta(T)$ and $\eta(E)$ to mark elements by the bulk criterion.
For convenience we define $\eta(\TT)$ and  $\eta(\cS_\Gamma)$ by
\begin{align*}
  \eta(\TT)^2 := \sum_{T\in\TT} \eta(T)^2, \qquad
  \eta(\cS_\Gamma)^2 := \sum_{E\in\cS_\Gamma} \eta(E)^2.
\end{align*}

\subsection{Piecewise smooth solution with boundary layer (Section~\ref{sec:dpg})}\label{sec:examples:smooth}
We consider the domain $\Omega:=(-1,1)\times (0,1)$ and the manufactured solution
\begin{align*}
  u(x,y) := \begin{cases}
    -16x^2(1-x)y(1-y) & x\geq 0, \\
    2(x+1)^3 - 3(x+1)^2 +1 & x<0.
  \end{cases}
\end{align*}
This solution satisfies $u(x,y) = 0$ on the part of $\Gamma=\partial\Omega$ where $x\geq 0$,
and $\partial_{\nn_\Omega}u = 0$ on the part of $\Gamma$ where $x<0$.
We calculate $f:=-\Delta u + u$ and note that $u\in H^1(\Delta,\Omega)$.
Also note that $u(x,y)$ is smooth in both regions, $x>0$ and $x<0$.
Our initial mesh consists of $8$ congruent triangles. 
We solve~\eqref{eq:dpg:varineqdisc} for $\star \in \{0,\nn,s\}$ and plot the errors
$\err(u)$, $\err(\ssigma)$, $\err(\widehat u)$, and $\err(\widehat\sigma)$ for a sequence of uniformly refined
triangulations.
Moreover, in the case $\star=s$ we compare these error quantities with the reliable error estimator $\eta$.
The results are given in Figure~\ref{fig:smooth:traceNormal} for $\star=0$, $\star=\nn$ and Figure~\ref{fig:smooth:sym}
for $\star=s$. 
We observe optimal convergence rates $\OO(h^\alpha) = \OO( (\#\TT)^{-\alpha/2})$ with $\alpha=1$ for the error quantities.
This rate is visualized by a triangle.
In Figure~\ref{fig:smooth:sym} we see that also the estimator $\eta(\TT)$ converges with this rate
whereas $\eta(\cS_\Gamma)$ has a higher convergence rate of approximately $\alpha = 2.8$.

\begin{figure}[htb]
  \begin{center}
    \includegraphics[width=0.49\textwidth]{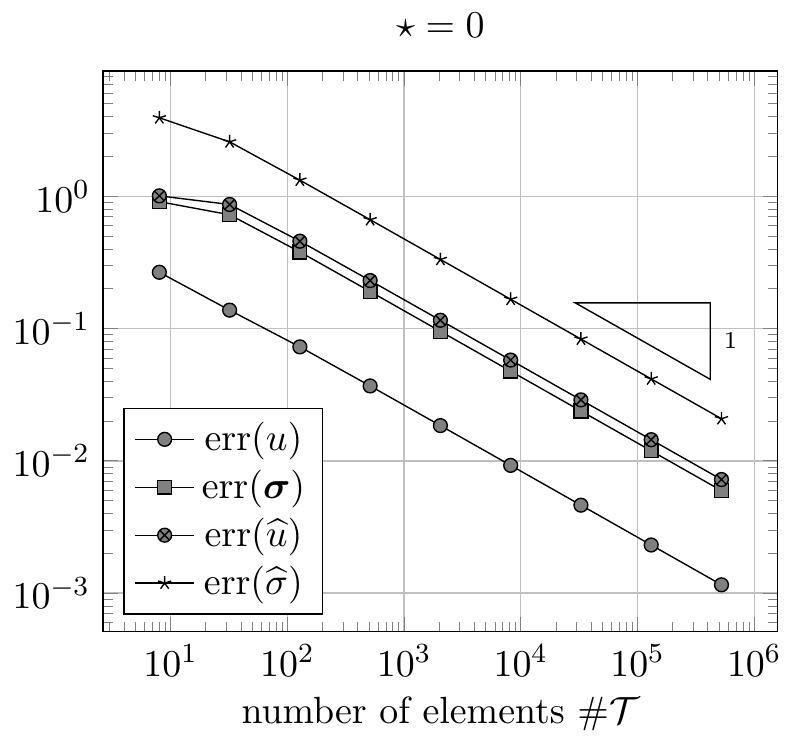}
    \includegraphics[width=0.49\textwidth]{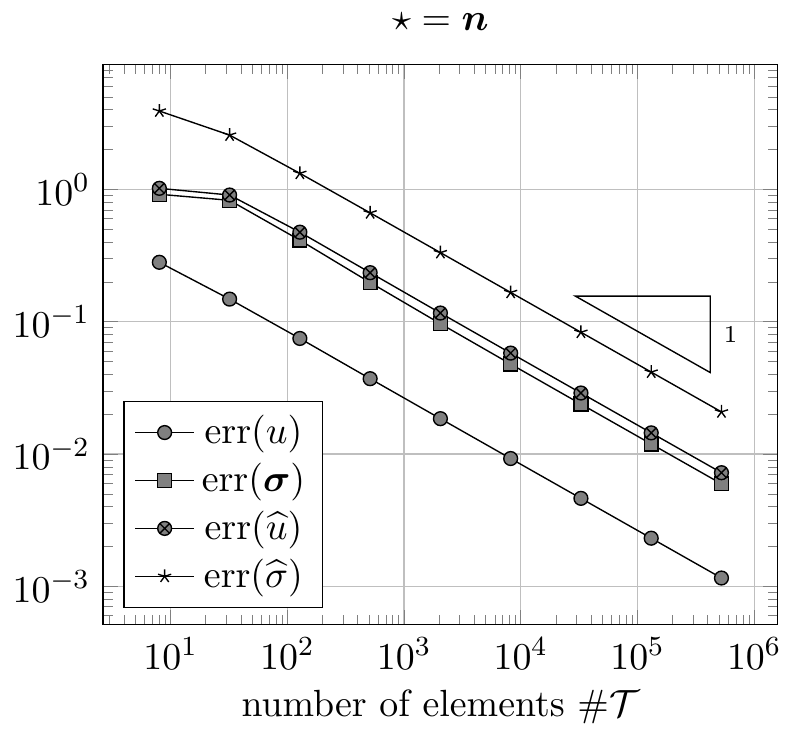}
  \end{center}
  \caption{Error quantities for Example~\ref{sec:examples:smooth} with $\star=0$ (left) and $\star=\nn$ (right).}
  \label{fig:smooth:traceNormal}
\end{figure}

\begin{figure}[htb]
  \begin{center}
    \includegraphics[width=0.65\textwidth]{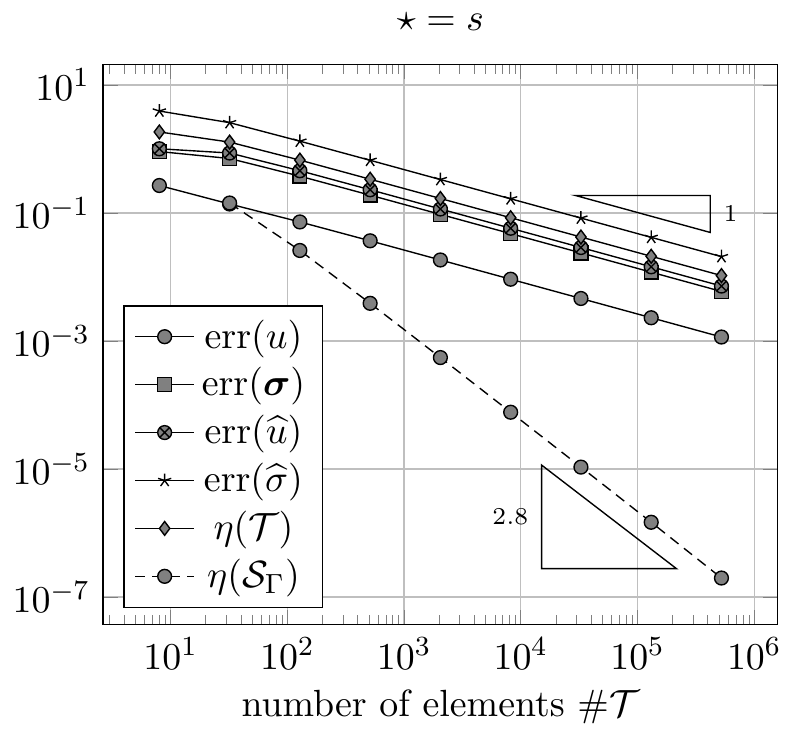}
  \end{center}
  \caption{Error quantities and estimators for Example~\ref{sec:examples:smooth} with $\star=s$.}
  \label{fig:smooth:sym}
\end{figure}

\subsection{Unknown solution (Section~\ref{sec:dpg})}\label{sec:examples:Lshape}
Let $\Omega = (-1,1)^2\setminus [-1,0]^2$ with initial triangulation visualized in Figure~\ref{fig:Lshape}.
We define
\begin{align*}
  f(x,y) := \begin{cases}
    -1 & |(x,y)|\leq 0.8, \\
    \tfrac12 & \text{else}.
  \end{cases}
\end{align*}
For this right-hand side the solution is not known to us in analytical form.
Therefore, we only compute the error estimators. The results are plotted in Figure~\ref{fig:Lshape:results}.
We observe that uniform refinement leads to a reduced order of convergence $\OO( (\#\TT)^{-\alpha/2})$
of approximately $\alpha=0.7$, whereas adaptive refinement regains the optimal order $\alpha=1$.
This is a strong indicator that the unknown solution has a singularity at the reentrant corner which
is what one expects.
Figure~\ref{fig:Lshape:meshes} visualizes meshes at different steps of the adaptive loop and supports this
observation.

\begin{figure}[htb]
  \begin{center}
    \includegraphics[width=0.6\textwidth]{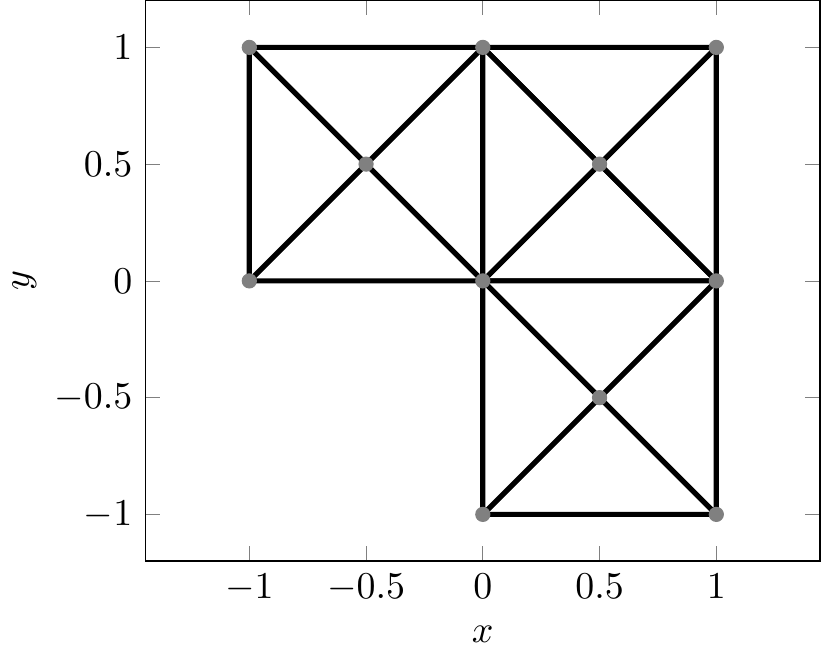}
  \end{center}
  \caption{L-shaped domain with initial triangulation of $12$ elements.}
  \label{fig:Lshape}
\end{figure}

\begin{figure}[htb]
  \begin{center}
    \includegraphics[width=0.65\textwidth]{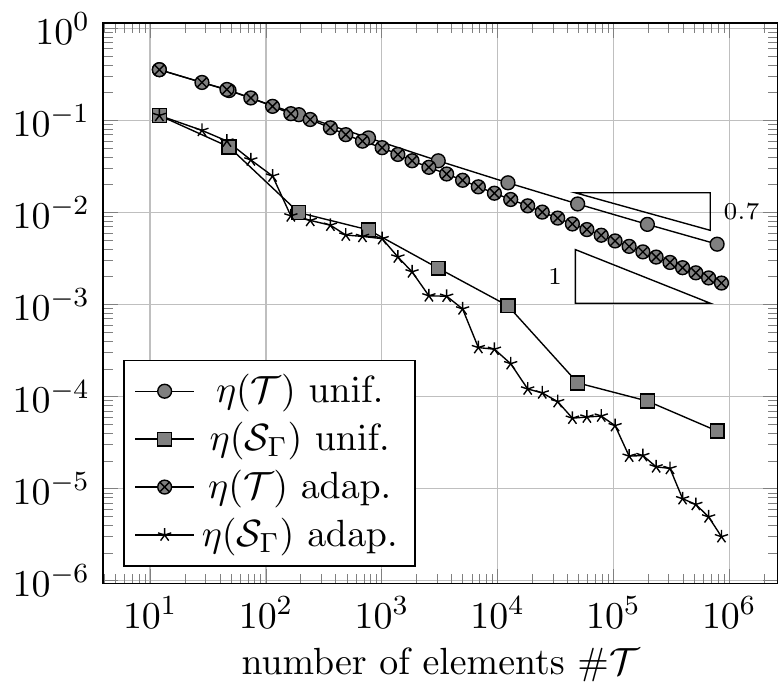}
  \end{center}
  \caption{Estimators for the example from Section~\ref{sec:examples:Lshape} for uniform and adaptive refinement.}
  \label{fig:Lshape:results}
\end{figure}

\begin{figure}[htb]
  \begin{center}
    \includegraphics[width=0.49\textwidth]{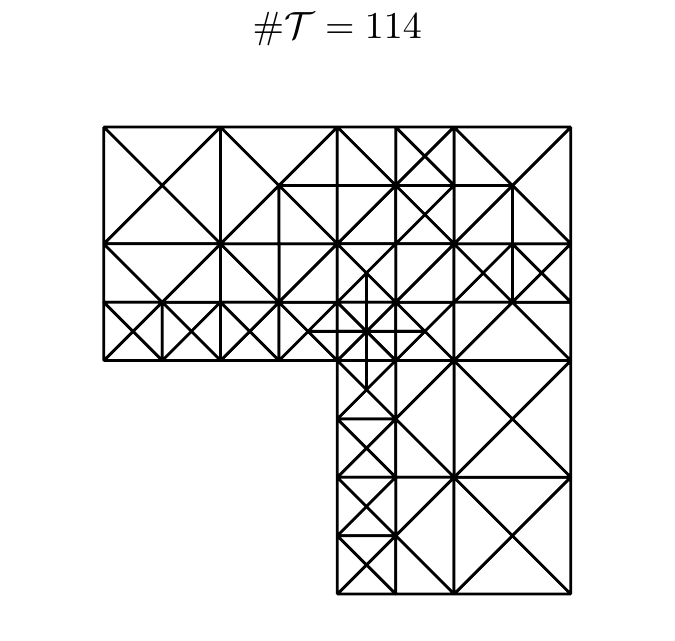}
    \includegraphics[width=0.49\textwidth]{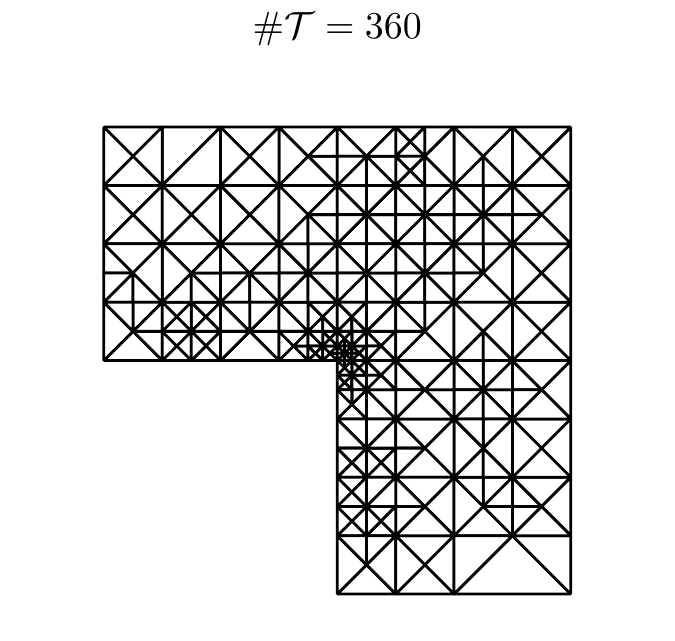}
    \includegraphics[width=0.49\textwidth]{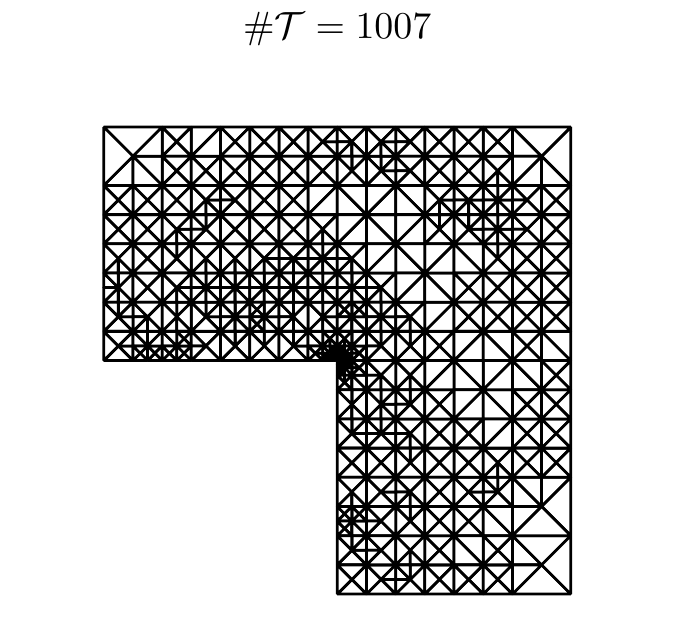}
    \includegraphics[width=0.49\textwidth]{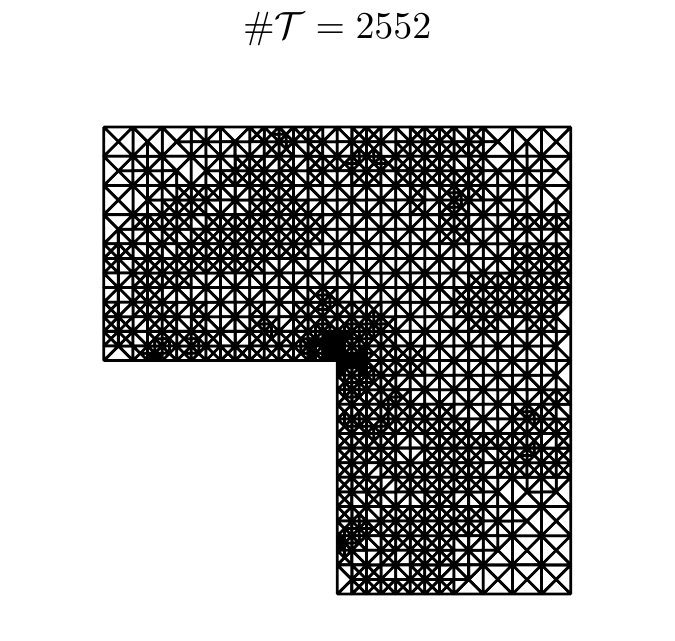}
  \end{center}
  \caption{Meshes at different steps in the adaptive algorithm for the problem from Section~\ref{sec:examples:Lshape}.}
  \label{fig:Lshape:meshes}
\end{figure}

\subsection{Piecewise smooth solution with boundary layer (Section~\ref{sec:sp})}\label{sec:examples:RDsmooth}

Let $\Omega:=(-1,1)\times (0,1)$ with manufactured solution
\begin{align*}
  u(x,y) := \begin{cases}
    -x^2\left(e^{-2(1-x)/\sqrt\eps}\right) 
    \left( e^{-y/\sqrt\eps} + e^{-(1-y)/\sqrt\eps} - e^{-1/\sqrt\eps} - 1\right)
    &x\geq 0, \\
    2(x+1)^3 - 3(x+1)^2 +1 & x<0.
  \end{cases}
\end{align*}
We choose $f=-\eps\Delta u+u$ so that $u$ satisfies~\eqref{eq:intro:sigpde} with $c=\eps$.
In particular, $u$ has a layer of order $\sqrt{\eps}$ at the boundary for $x\geq 0$.
We solve the variational inequality~\eqref{eq:sp:varineqdisc} on a sequence of adaptively refined triangulations for
$\eps\in\{10^{-2},10^{-4},10^{-6},10^{-8}\}$. We start with a coarse initial triangulation that consists of
only $\#\TT_0 = 8$ congruent triangles.

In Figure~\ref{fig:sp:smooth} we compare the estimator $\eta$ with the total error $\err(\uu)$. As in the
works~\cite{HeuerK_RDM,FuehrerH_RCD}, we observe that, after boundary layers have been resolved, our method leads
to an optimal convergence rate $\OO( (\#\TT)^{-\alpha/2})$ ($\alpha=1$).
We also observe that the curves representing $\eta(\TT)$ and $\err(\uu)$ are quite close together, even
in the pre-asymptotic range, and uniformly in $\eps$ (for the selected values). This confirms the robustness
of the a posteriori estimate by Theorem~\ref{thm:sp:aposteriori}. We also
see that the boundary estimator $\eta(\cS_\Gamma)$ is small in comparison to $\eta(\TT)$, 
and has a higher convergence rate.

\begin{figure}[htb]
  \begin{center}
    \includegraphics[width=0.49\textwidth]{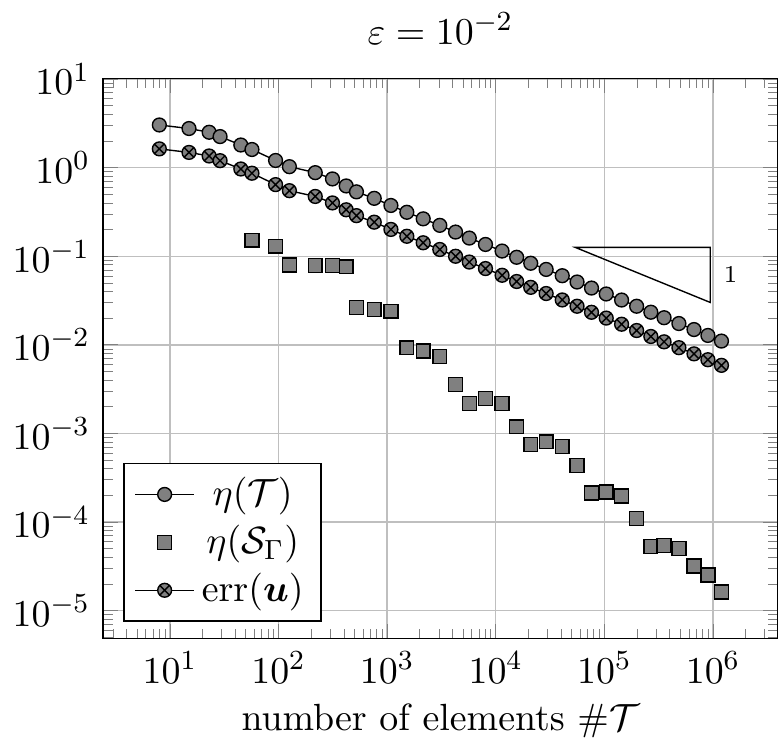}
    \includegraphics[width=0.49\textwidth]{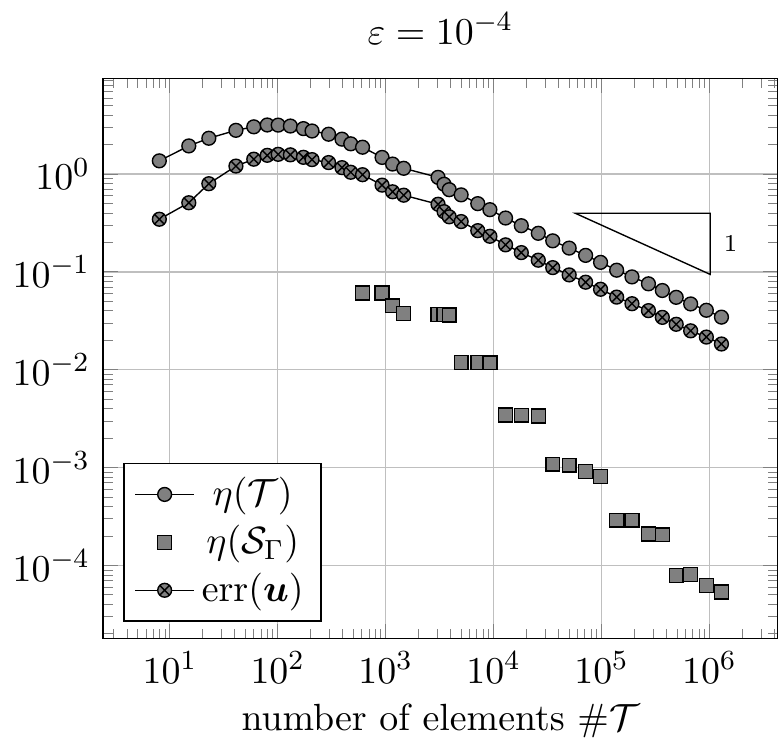}
    \includegraphics[width=0.49\textwidth]{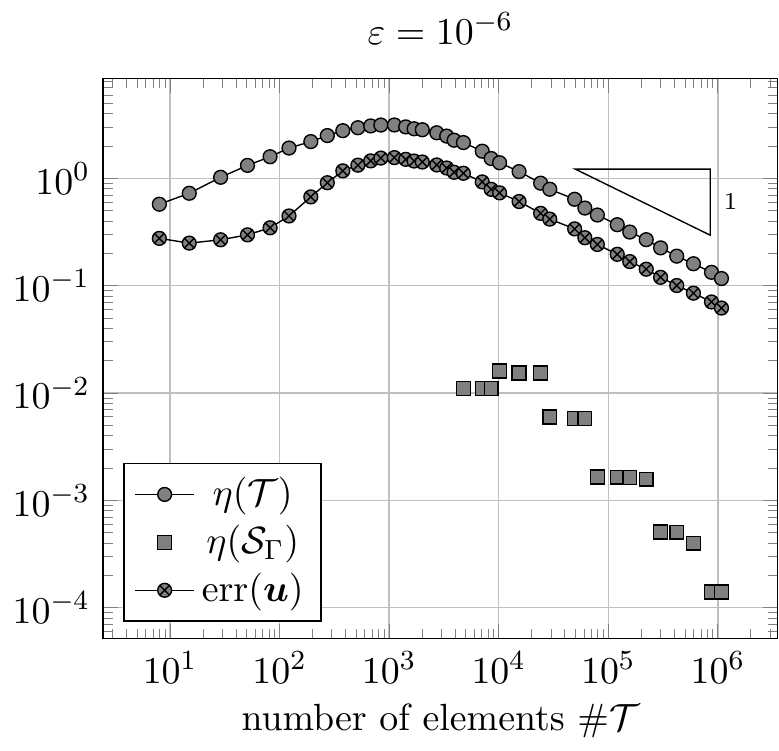}
    \includegraphics[width=0.49\textwidth]{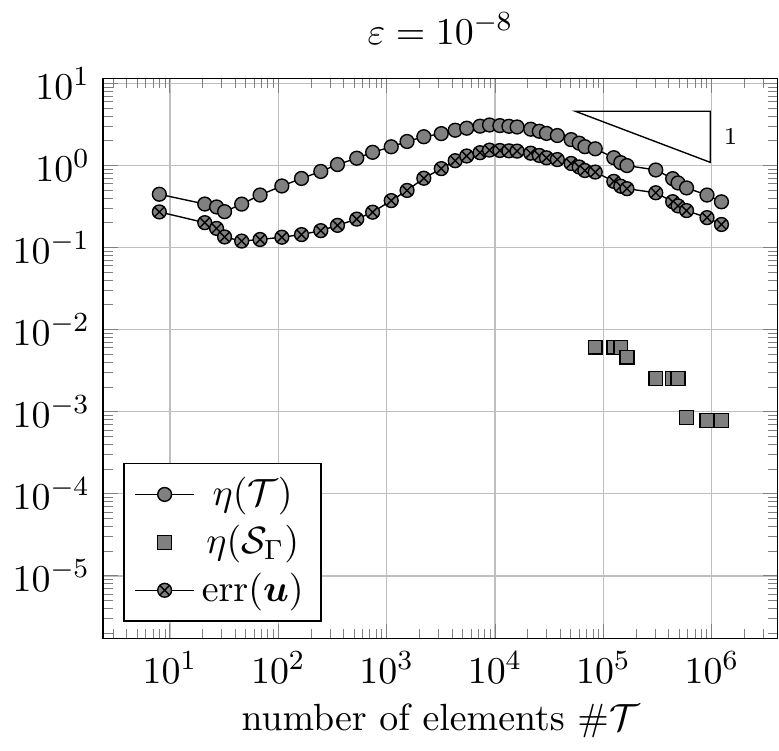}
  \end{center}
  \caption{Comparison of total error $\err(\uu)$ and estimators $\eta(\TT)$, $\eta(\cS_\Gamma)$ for 
    the problem from Section~\ref{sec:examples:RDsmooth}.}
  \label{fig:sp:smooth}
\end{figure}

Our primary interest was to construct a robust method in the sense that
the error in the balanced norm of the field variables $u,\ssigma,\rho$ is controlled uniformly in $\eps$.
The error estimator $\eta$ does precisely this, as stated by Theorem~\ref{thm:sp:aposteriori}
and seen in Figure~\ref{fig:sp:smooth}. To further underline this statement,
Figure~\ref{fig:sp:ratio} shows the ratio $(\err(u)^2+\err(\ssigma)^2+\err(\rho)^2)^{1/2}/\eta(\TT)$. 
We observe that, again after boundary layers have been resolved, this ratio is between $0.5$ and $0.55$
uniformly with respect to $\eps$. In fact, this number is
close to $1/\sqrt{\beta} = 1/\sqrt{3}\sim 0.5774$.
Note that, by the product structure of $V$,
\(
   \norm{B\uu_h-L}{V'}^2 = \sum_{T\in\TT} \norm{B\uu_h-L}{V'(T)}^2,
\)
so that $\eta(\TT)=\sqrt{\beta}\norm{B\uu_h-L}{V'}$, cf.~~\eqref{eq:aposteriori:estdef}.
Our numerical results therefore indicate that the slight overestimation of the error by $\eta$
is due to the choice of $\beta$, and is (asymptotically) uniform in $\eps$.

\begin{figure}[htb]
  \begin{center}
    \includegraphics[width=0.5\textwidth]{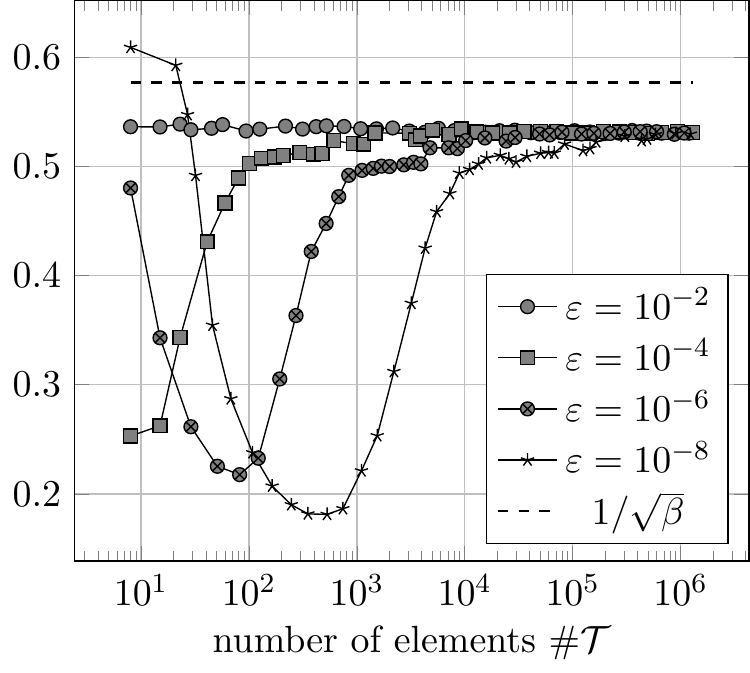}
  \end{center}
  \caption{Ratio $(\err(u)^2+\err(\ssigma)^2+\err(\rho)^2)^{1/2}/\eta(\TT)$ for the problem from
  Section~\ref{sec:examples:RDsmooth}.}
  \label{fig:sp:ratio}
\end{figure}

\bibliographystyle{abbrv}
\bibliography{bib,heuer}

\begin{thebibliography}{10}

\bibitem{AttiaCS_09_FOS}
F.~S. Attia, Z.~Cai, and G.~Starke.
\newblock First-order system least squares for the {S}ignorini contact problem
  in linear elasticity.
\newblock {\em SIAM J. Numer. Anal.}, 47(4):3027--3043, 2009.

\bibitem{BrezziHR_77_EEF}
F.~Brezzi, W.~W. Hager, and P.-A. Raviart.
\newblock Error estimates for the finite element solution of variational
  inequalities.
\newblock {\em Numer. Math.}, 28(4):431--443, 1977.

\bibitem{BrezziHR_78_EEF}
F.~Brezzi, W.~W. Hager, and P.-A. Raviart.
\newblock Error estimates for the finite element solution of variational
  inequalities. {II}. {M}ixed methods.
\newblock {\em Numer. Math.}, 31(1):1--16, 1978/79.

\bibitem{BroersenS_14_RPG}
D.~Broersen and R.~Stevenson.
\newblock A robust {Petrov}-{Galerkin} discretisation of convection-diffusion
  equations.
\newblock {\em Comput. Math. Appl.}, 68(11):1605--1618, 2014.

\bibitem{BustinzaS_12_EEL}
R.~Bustinza and F.-J. Sayas.
\newblock Error estimates for an {LDG} method applied to {S}ignorini type
  problems.
\newblock {\em J. Sci. Comput.}, 52(2):322--339, 2012.

\bibitem{CarstensenG_97_FBC}
C.~Carstensen and J.~Gwinner.
\newblock F{EM} and {BEM} coupling for a nonlinear transmission problem with
  {S}ignorini contact.
\newblock {\em SIAM J. Numer. Anal.}, 34(5):1845--1864, 1997.

\bibitem{ChanHBTD_14_RDM}
J.~Chan, N.~Heuer, T.~Bui-Thanh, and L.~Demkowicz.
\newblock Robust {DPG} method for convection-dominated diffusion problems {II}:
  {Adjoint} boundary conditions and mesh-dependent test norms.
\newblock {\em Comput. Math. Appl.}, 67(4):771--795, 2014.

\bibitem{ChoulyH_13_NMU}
F.~Chouly and P.~Hild.
\newblock A {N}itsche-based method for unilateral contact problems: numerical
  analysis.
\newblock {\em SIAM J. Numer. Anal.}, 51(2):1295--1307, 2013.

\bibitem{DemkowiczG_11_ADM}
L.~Demkowicz and J.~Gopalakrishnan.
\newblock Analysis of the {DPG} method for the {Poisson} problem.
\newblock {\em SIAM J. Numer. Anal.}, 49(5):1788--1809, 2011.

\bibitem{DemkowiczG_11_CDP}
L.~Demkowicz and J.~Gopalakrishnan.
\newblock A class of discontinuous {Petrov-Galerkin} methods. {Part II}:
  {O}ptimal test functions.
\newblock {\em Numer. Methods Partial Differential Eq.}, 27:70--105, 2011.

\bibitem{DemkowiczH_13_RDM}
L.~Demkowicz and N.~Heuer.
\newblock Robust {DPG} method for convection-dominated diffusion problems.
\newblock {\em SIAM J. Numer. Anal.}, 51(5):2514--2537, 2013.

\bibitem{DrouetH_15_OCD}
G.~Drouet and P.~Hild.
\newblock Optimal convergence for discrete variational inequalities modelling
  {S}ignorini contact in 2{D} and 3{D} without additional assumptions on the
  unknown contact set.
\newblock {\em SIAM J. Numer. Anal.}, 53(3):1488--1507, 2015.

\bibitem{Falk_74_EEA}
R.~S. Falk.
\newblock Error estimates for the approximation of a class of variational
  inequalities.
\newblock {\em Math. Comput.}, 28:963--971, 1974.

\bibitem{Fichera_71_UCE}
G.~Fichera.
\newblock Unilateral constraints in elasticity.
\newblock In {\em Actes du {C}ongr\`es {I}nternational des {M}ath\'ematiciens
  ({N}ice, 1970), {T}ome 3}, pages 79--84. Gauthier-Villars, Paris, 1971.

\bibitem{FuehrerH_RCD}
T.~F{\"u}hrer and N.~Heuer.
\newblock Robust coupling of {DPG} and {BEM} for a singularly perturbed
  transmission problem.
\newblock {arXiv}: 1603.05164, 2016.

\bibitem{FuehrerHK_CDB}
T.~F{\"u}hrer, N.~Heuer, and M.~Karkulik.
\newblock On the coupling of {DPG} and {BEM}.
\newblock {arXiv}: 1508.00630, 2015.
\newblock Accepted for publication in {\em Math. Comp.}

\bibitem{FuehrerHS_TSD}
T.~F{\"u}hrer, N.~Heuer, and J.~{Sen Gupta}.
\newblock A time-stepping {DPG} scheme for the heat equation.
\newblock {arXiv}: 1607.00301, 2016.

\bibitem{GaticaMS_11_NAT}
G.~N. Gatica, M.~Maischak, and E.~P. Stephan.
\newblock Numerical analysis of a transmission problem with {S}ignorini contact
  using mixed-{FEM} and {BEM}.
\newblock {\em ESAIM Math. Model. Numer. Anal.}, 45(4):779--802, 2011.

\bibitem{Glowinski_08_NMN}
R.~Glowinski.
\newblock {\em Numerical methods for nonlinear variational problems}.
\newblock Scientific Computation. Springer-Verlag, Berlin, 2008.
\newblock Reprint of the 1984 original.

\bibitem{GlowinskiLT_81_NAV}
R.~Glowinski, J.-L. Lions, and R.~Tr{\'e}moli{\`e}res.
\newblock {\em Numerical analysis of variational inequalities}, volume~8 of
  {\em Studies in Mathematics and its Applications}.
\newblock North-Holland Publishing Co., Amsterdam-New York, 1981.
\newblock Translated from the French.

\bibitem{GopalakrishnanQ_14_APD}
J.~Gopalakrishnan and W.~Qiu.
\newblock An analysis of the practical {DPG} method.
\newblock {\em Math. Comp.}, 83(286):537--552, 2014.

\bibitem{HeuerK_RDM}
N.~Heuer and M.~Karkulik.
\newblock A robust {DPG} method for singularly perturbed reaction-diffusion
  problems.
\newblock {arXiv}: 1509.07560, 2015.

\bibitem{HlavacekHNL_88_SVI}
I.~Hlav{\'a}{\v{c}}ek, J.~Haslinger, J.~Ne{\v{c}}as, and
  J.~Lov{\'{\i}}{\v{s}}ek.
\newblock {\em Solution of variational inequalities in mechanics}, volume~66 of
  {\em Applied Mathematical Sciences}.
\newblock Springer-Verlag, New York, 1988.
\newblock Translated from the Slovak by J. Jarn{\'{\i}}k.

\bibitem{HlavacekL_77_FEM}
I.~Hlav{\'a}{\v{c}}ek and J.~Lov{\'{\i}}{\v{s}}ek.
\newblock A finite element analysis for the {S}ignorini problem in plane
  elastostatics.
\newblock {\em Apl. Mat.}, 22(3):215--228, 1977.

\bibitem{HoppeK_94_AMM}
R.~H.~W. Hoppe and R.~Kornhuber.
\newblock Adaptive multilevel methods for obstacle problems.
\newblock {\em SIAM J. Numer. Anal.}, 31(2):301--323, 1994.

\bibitem{KarkkainenKT_03_ALA}
T.~K{\"a}rkk{\"a}inen, K.~Kunisch, and P.~Tarvainen.
\newblock Augmented {L}agrangian active set methods for obstacle problems.
\newblock {\em J. Optim. Theory Appl.}, 119(3):499--533, 2003.

\bibitem{KikuchiO_88_CPE}
N.~Kikuchi and J.~T. Oden.
\newblock {\em Contact problems in elasticity: a study of variational
  inequalities and finite element methods}, volume~8 of {\em SIAM Studies in
  Applied Mathematics}.
\newblock Society for Industrial and Applied Mathematics (SIAM), Philadelphia,
  PA, 1988.

\bibitem{LinS_12_BFE}
R.~Lin and M.~Stynes.
\newblock A balanced finite element method for singularly perturbed
  reaction-diffusion problems.
\newblock {\em SIAM J. Numer. Anal.}, 50(5):2729--2743, 2012.

\bibitem{Loppin_78_PCS}
G.~Loppin.
\newblock Poin\c connement d'un support \'elastique et applications
  num\'eriques.
\newblock {\em J. M\'ecanique}, 17(3):455--479, 1978.

\bibitem{MaischakS_05_FBC}
M.~Maischak and E.~P. Stephan.
\newblock A {FEM}-{BEM} coupling method for a nonlinear transmission problem
  modelling {C}oulomb friction contact.
\newblock {\em Comput. Methods Appl. Mech. Engrg.}, 194(2-5):453--466, 2005.

\bibitem{NiemiCC_13_ASD}
A.~H. Niemi, N.~O. Collier, and V.~M. Calo.
\newblock Automatically stable discontinuous {P}etrov-{G}alerkin methods for
  stationary transport problems: quasi-optimal test space norm.
\newblock {\em Comput. Math. Appl.}, 66(10):2096--2113, 2013.

\bibitem{Rodrigues_87_OPM}
J.-F. Rodrigues.
\newblock {\em Obstacle problems in mathematical physics}, volume 134 of {\em
  North-Holland Mathematics Studies}.
\newblock North-Holland Publishing Co., Amsterdam, 1987.
\newblock Notas de Matem{\'a}tica [Mathematical Notes], 114.

\bibitem{ScarpiniV_77_EEA}
F.~Scarpini and M.~A. Vivaldi.
\newblock Error estimates for the approximation of some unilateral problems.
\newblock {\em RAIRO Anal. Num\'er.}, 11(2):197--208, 221, 1977.

\bibitem{Signorini_33_SAQ}
A.~Signorini.
\newblock Sopra alcune questioni di statica dei sistemi continui.
\newblock {\em Ann. Scuola Norm. Sup. Pisa Cl. Sci. (2)}, 2(2):231--251, 1933.

\end{thebibliography}
\end{document}